\documentclass[letterpaper,oneside,10pt]{article}
\usepackage[letterpaper,left=1in,right=1in,top=1.5in,bottom=1.5in]{geometry}
\title{Counterexamples to Hochschild--Kostant--Rosenberg in characteristic $p$}
\author{Benjamin Antieau, Bhargav Bhatt, and Akhil Mathew}
\date{\today}

\usepackage{enumerate}

\newcommand{\Fil}{\mathrm{Fil}}
\newcommand{\ShK}{\mathbf{K}^{\et}(-)_{\mathcal{A}}}
\newcommand{\ShHH}{\mathbf{HH}(-/k)_{\mathcal{A}}}

\usepackage[pdfstartview=FitH,
            pdfauthor={Antieau, Bhatt, and Mathew},
            pdftitle={HKR in characteristic p},
            colorlinks,
            linkcolor=reference,
            citecolor=citation,
            urlcolor=e-mail,
            ]{hyperref}

\usepackage{amsmath}
\usepackage{amscd}
\usepackage{amsbsy}
\usepackage{amssymb}
\usepackage{verbatim}
\usepackage{eufrak}
\usepackage{eucal}
\usepackage{microtype}
\usepackage{hyperref}
\usepackage{mathrsfs}
\usepackage{amsthm}
\usepackage{stmaryrd}
\usepackage{xy}
\input xy
\xyoption{all}
\usepackage{tikz}
\usetikzlibrary{matrix,arrows}
\usepackage{dsfont}
\usepackage{tikz-cd}
\usepackage{float}

\usepackage{fancyhdr}
\usepackage{color}
\definecolor{todo}{rgb}{1,0,0}
\definecolor{conditional}{rgb}{0,1,0}
\definecolor{e-mail}{rgb}{0,.40,.80}
\definecolor{reference}{rgb}{.20,.60,.22}
\definecolor{mrnumber}{rgb}{.80,.40,0}
\definecolor{citation}{rgb}{0,.40,.80}

\usepackage[bbgreekl]{mathbbol}
\usepackage{amsfonts}
\DeclareSymbolFontAlphabet{\mathbb}{AMSb} 
\DeclareSymbolFontAlphabet{\mathbbl}{bbold}

\let\oldmarginpar\marginpar
\renewcommand\marginpar[1]{\-\oldmarginpar[\raggedleft\footnotesize #1]%
{\raggedright\footnotesize #1}}

\newcommand{\Ascr}{\mathcal{A}}

\newcommand{\Cscr}{\mathcal{C}}

\newcommand{\Oscr}{\mathcal{O}}

\newcommand{\Xscr}{\mathcal{X}}
\newcommand{\Yscr}{\mathcal{Y}}

\newcommand{\B}{{B}}

\renewcommand{\H}{{H}}
\newcommand{\K}{\mathrm{K}}

\newcommand{\FF}{\mathds{F}}

\newcommand{\NN}{\mathds{N}}
\newcommand{\PP}{\mathds{P}}

\newcommand{\ZZ}{\mathds{Z}}

\newcommand{\gr}{\mathrm{gr}}

\newcommand{\crys}{\mathrm{crys}}

\newcommand{\dR}{\mathrm{dR}}

\newcommand{\Sp}{\mathrm{Sp}}

\newcommand{\HKR}{\mathrm{HKR}}

\DeclareMathOperator{\dlog}{dlog}

\renewcommand{\geq}{\geqslant}
\renewcommand{\leq}{\leqslant}

\newcommand{\HC}{\mathrm{HC}}
\newcommand{\THH}{\mathrm{THH}}
\newcommand{\HP}{\mathrm{HP}}
\newcommand{\TP}{\mathrm{TP}}

\newcommand{\HH}{\mathrm{HH}}

\DeclareMathOperator{\Br}{Br}

\newcommand{\Perf}{\mathrm{Perf}}

\newcommand{\PGL}{\mathbf{PGL}}

\newcommand{\GL}{\mathbf{GL}}

\newcommand{\Gm}{\mathds{G}_{m}}

\newcommand{\et}{\mathrm{\acute{e}t}}

\newcommand{\fppf}{\mathrm{fppf}}

\newcommand{\per}{\mathrm{per}}

\newcommand{\ind}{\mathrm{ind}}

\DeclareMathOperator{\Spec}{Spec}

\newcommand{\we}{\simeq}
\newcommand{\iso}{\cong}

\newcommand{\cosimp}[3]{\xymatrix@1{#1 \ar@<.4ex>[r] \ar@<-.4ex>[r] & {\ }#2 \ar@<0.8ex>[r] \ar[r] \ar@<-.8ex>[r] & {\ } #3 \ar@<1.2ex>[r] \ar@<.4ex>[r] \ar@<-.4ex>[r] \ar@<-1.2ex>[r] & \cdots }}

\usepackage{cleveref}
\theoremstyle{plain}
\newtheorem{theorem}{Theorem}[section]
\newtheorem*{theorem*}{Theorem}
\newtheorem{lemma}[theorem]{Lemma}

\newtheorem{proposition}[theorem]{Proposition}

\newtheorem{corollary}[theorem]{Corollary}
\newtheorem*{corollary*}{Corollary}

\theoremstyle{plain}

\theoremstyle{definition}

\newtheoremstyle{named}{}{}{\itshape}{}{\bfseries}{.}{.5em}{#1 \thmnote{#3}}
\theoremstyle{named}

\theoremstyle{definition}
\newtheorem{definition}[theorem]{Definition}
\newtheorem{warning}[theorem]{Warning}

\newtheorem{notation}[theorem]{Notation}
\newtheorem{observation}[theorem]{Observation}

\newtheorem{example}[theorem]{Example}
\newtheorem*{example*}{Example}

\newtheorem*{question*}{Question}
\newtheorem{construction}[theorem]{Construction}
\newtheorem{calculation}[theorem]{Calculation}
\newtheorem{remark}[theorem]{Remark}

\renewcommand{\bigwedge}{\wedge}
\renewcommand{\mathbb}[1]{\mathbf{#1}}
\renewcommand{\mathds}[1]{\mathbf{#1}}
\renewcommand{\mathscr}[1]{\mathcal{#1}}

\begin{document}

\maketitle

\begin{abstract}
    \noindent
    We give counterexamples to the degeneration of the HKR spectral sequence in
    characteristic $p$, both in the untwisted and twisted settings. We also
    prove that the de Rham--$\HP$ and crystalline--$\TP$ spectral sequences
    need not degenerate. 


    \paragraph{Mathematics Subject Classification 2010.}
    \href{http://www.ams.org/mathscinet/msc/msc2010.html?t=13Dxx&btn=Current}{13D03},
    \href{http://www.ams.org/mathscinet/msc/msc2010.html?t=14Fxx&btn=Current}{14F40},
    \href{http://www.ams.org/mathscinet/msc/msc2010.html?t=16Exx&btn=Current}{16E40},
    \href{http://www.ams.org/mathscinet/msc/msc2010.html?t=19Dxx&btn=Current}{19D55}.
\end{abstract}


\section{Introduction}

Let $k$ be a commutative  ring and $R$ a commutative flat $k$-algebra. Recall
that the Hochschild homology complex $\HH(R/k)$ of $R$ relative to $k$ can be defined as the ``functions on the self-intersection of the diagonal of $\mathrm{Spec}(R) \to \mathrm{Spec}(k)$'', i.e., as 
\[ \HH(R/k) = R \otimes^L_{R \otimes_k R} R.\]
The Hochschild--Kostant--Rosenberg (HKR) theorem connects the cohomology groups
of this complex to differential forms over $k$: if $R$ is smooth, there are canonical isomorphisms 
\begin{equation} 
\label{HKRthm} 
 H^{-n}(\HH(R/k)) \iso
\Omega^n_{R/k}. 
\end{equation}

In characteristic zero, the formula \eqref{HKRthm} upgrades to a canonical decomposition
(often referred to as the Hodge decomposition) of the Hochschild complex, cf.
for instance
\cite[Th. 8.6]{Qui70}, \cite{GS87}, \cite{Wei97};
this implies that for 
a smooth variety $X/k$ there are canonical isomorphisms
\begin{equation}\label{eq:hkr}
    H^{-n}(\HH(X/k))\iso\bigoplus_{s-t=n} H^t(X,\Omega^s_{X/k}).
\end{equation}

In this paper, we study to  what extent the decomposition \eqref{eq:hkr} might hold when  $k$ has positive characteristic $p > 0$. More precisely, given a smooth $k$-scheme $X$, the canonicity of the HKR isomorphism \eqref{HKRthm} already implies that one has an $E_2$-spectral sequence, which we call the HKR spectral sequence, of the form
\[ E_2^{s,t} = H^s(X,\Omega^{-t}_{X/k}) \Rightarrow H^{s+t} \HH(X/k).\]
  We ask here whether this spectral sequence always has to degenerate at $E_2$ (e.g., for smooth proper varieties). 

Degeneration of the HKR spectral sequence is known if $X/k$ is smooth proper of dimension $\leq p$  by work of 
Yekutieli~\cite{Yek02} (in the case where $\dim(X) < p$, in which case one also gets a
canonical decomposition) and by Antieau--Vezzosi~\cite{antieau-vezzosi} (which allows $\dim(X) = p$ as well). Additionally, Rao--Yang--Yang--Yu~\cite{ryyy} recently proved that HKR holds for
the blowup of a smooth proper $X$ along a smooth closed subscheme $Z$ if and only if it holds
for $X$ and $Z$. 
There are additional examples in which one can verify degeneration, e.g.,
smooth complete intersections in projective space  (see \cite[Ex. 1.7]{antieau-vezzosi}). 
In contrast, we prove the following result.

\begin{theorem}\label{thm:intro1}
    Let $k$ be a perfect field of characteristic $p>0$. There exists a smooth projective $2p$-dimensional $k$-scheme $X$ such that the HKR spectral sequence for $X$ does not degenerate, so there can be no Hodge decomposition of $\HH(X/k)$. Specifically, we construct such an $X$ for which the differential $d_p \colon H^0(X,\Omega^1_{X/k}) \to H^p(X,\wedge^p L_{X/k})$ in the HKR spectral sequence is nonzero.
\end{theorem}

We also give related examples where the de Rham--$\HP$ and crystalline--$\TP$
spectral sequences constructed
in~\cite{BMS2} are non-degenerate and examples where the crystalline--$\TP$ 
spectral sequence gives a non-split filtration on $H^*(\TP(X))$. For details, see
Theorem~\ref{thm:other}.

Our method is to understand these spectral sequences in the case where we
replace the scheme $X$ by the classifying stack $BG$ of a group scheme $G$. In
fact, slightly surprisingly, $G = \mu_p$ already leads to
Theorem~\ref{thm:intro1}. In this case, the reason for non-degeneracy of the
HKR spectral sequence is relatively easy to describe, at least informally: the
Hochschild homology of $B\mu_p$ is concentrated in degree $0$ (as the category
of quasicoherent sheaves on $B\mu_p$---or equivalently the category of
representations of $\mu_p$---is just the $p$-fold direct sum of the category
of vector spaces), while the Hodge cohomology of $B\mu_p$ is not concentrated
in degree $0$ (as there are non-trivial $1$-forms on $B\mu_p$ arising from the
singularities of $\mu_p$ as a scheme).

To pass from the stacks discussed above to the examples of
Theorem~\ref{thm:intro1}, we approximate $BG$ by smooth projective varieties,
i.e., we find maps $X \to BG$ with $X$ smooth projective such that the
pullback map on various cohomology theories considered above is injective.
More specifically, we prove the following theorem.

\begin{theorem}\label{thm:approximationintro}
    Suppose that $k$ is a perfect field  of characteristic $p > 0$ and that $G$
	is an affine $k$-group scheme which
    is either finite or geometrically reductive. For any integer $d \geq 0$, there exists a
    smooth projective $k$-scheme $X$ of dimension $d$ together with a map
    $X\rightarrow BG$ such that the pullback $H^s(BG, \bigwedge^t L_{BG/k}) \to H^s(X, \bigwedge^t L_{X/k})$ is injective for $s+t \leq d$.
    
\end{theorem}

The geometric idea behind finding such approximations goes back to the work of
Godeaux and Serre. However, as the relevant group schemes $G$ are not smooth,
one runs into possibly singular complete intersections in projective space as
intermediate objects in this argument. To handle their cohomology, we prove a
version of the weak Lefschetz theorem for Hodge cohomology for such complete
intersections.

We also consider the twisted version of this question. 
Let $X/k$ be a smooth $k$-scheme and let $\alpha\in\H^2(X,\Gm)$ be a cohomological Brauer class. One constructs a twisted form
$\HH(X/k,\alpha)$ of Hochschild homology as in the study of twisted $K$-theory.
In fact, if $\alpha$ is the Brauer class of an Azumaya algebra $\Ascr$, then
$\HH(X/k,\alpha)\we\HH(\Ascr/k)$.
By work of Corti\~{n}as--Weibel~\cite{cortinas-weibel}, there is a spectral sequence
$$E_2^{s,t}=H^s(X,\Omega^{-t}_{X/k})\Rightarrow H^{s+t}(\HH(X/k,\alpha)).$$
We call this the $\alpha$-twisted HKR spectral sequence; when $\alpha=0$, it is
the HKR spectral sequence. In general, the terms of the $E_2$-page are the
same as in the untwisted case,
but the differentials might be different. When $k$ is a field and $X$ is
additionally proper over $k$, the degeneration of the
HKR spectral sequence is equivalent to the existence of an isomorphism as
in~\eqref{eq:hkr}.

\begin{theorem}\label{thm:intro2}
    Let $k$ be a field of characteristic $p>0$. There exists a smooth
    projective threefold $X$ over $k$ and a Brauer class $\alpha\in\Br(X)$ such that the
    $\alpha$-twisted HKR spectral sequence does not degenerate. Specifically, we construct such an $X$ for which the
    differential
    $d_2^\alpha\colon H^0(X,\Omega^0_{X/k})\rightarrow H^2(X,\Omega^1_{X/k})$ in the $\alpha$-twisted HKR spectral sequence is
    nonzero.
\end{theorem}

For $p=2$, we can do a little better and find surface examples. In fact, classical Enriques surfaces work in that case (Propositions~\ref{prop:differential} and \ref{prop:enriques}).

Theorem~\ref{thm:intro2} gives examples of a smooth projective $k$-schemes $X$ and
Azumaya algebras $\Ascr$ on $X$ such that  $\dim_k H^i(\HH(\Ascr/k)) \neq \dim_k H^i(\HH(X/k))$. This is in contrast to the
affine case of Corti\~{n}as--Weibel~\cite{cortinas-weibel} and shows that their theorem cannot be
globalized. These examples also lead to cases of $\PP^n$-bundles $P\rightarrow X$ such that
pullback kills Hodge and de Rham cohomology classes, a phenomenon which can
only exist in characteristic $p$. See Section~\ref{sec:conic}.

\medskip
\noindent
{\bf Conventions.} Throughout, we use cohomological indexing conventions. 
Given a commutative ring $k$, we will let $D(k)$ denote the derived $\infty$-category
of $k$-modules, and $\mathrm{Sp}$ denote the $\infty$-category of spectra. For the purposes of this paper, a $d$-dimensional scheme is by definition equidimensional. Moreover, when we say that a spectral sequence discussed in \S \ref{sec:ss} degenerates without specifying a page, the degeneration is always intended to begin at the first page where the spectral sequence is defined; when we assert that a spectral sequence does not degenerate, we specify the nonzero page and differential.

\medskip
\noindent
{\bf Acknowledgments.} We would like to thank Daniel Bragg, Lukas Brantner, Johan de Jong, 
H{\'e}l{\`e}ne Esnault,
Marci Hablicsek,
Daniel Halpern-Leistner,
Yank\i\, Lekili, Alexander Petrov, Bertrand To\"en, Burt Totaro, and Gabriele Vezzosi for discussions about
HKR in characteristic $p$.

This material is partially based upon work supported by the National Science Foundation
under Grant No. DMS-1440140 while the authors were in residence at the
Mathematical Sciences Research Institute in Berkeley, California, during the
Spring 2019 semester.
The first author was partially supported by NSF Grant DMS-1552766.
The second author was partially supported by the NSF grants \#1501461 and \#1801689, a Packard fellowship and the Simons Foundation grant \#622511.
This work was done while the third author was a Clay Research Fellow.

\section{Hochschild homology and de Rham cohomology of
stacks}\label{sec:invariants}

Fix a commutative ring $k$. In this section, we introduce the various cohomology
theories that we shall use later in the context of algebraic stacks over $k$.
Our strategy is to define the cohomology of a stack via descent. To simplify
definitions and avoid subtleties, we work with the syntomic topology and stick
to stacks which are themselves syntomic (as defined next); this includes all
smooth (or even lci) algebraic stacks over $k$, which is sufficient for our purposes.

\begin{notation}
A map $R \to S$ of commutative rings is called a {\bf syntomic map} if it is
flat and finitely presented with $L_{S/R} \in D(S)$ having Tor amplitude in
$[-1,0]$; there is a similar definition for maps of schemes.  Let
$\mathrm{Syn}_k$ denote the category of syntomic $k$-algebras; its opposite
category $\mathrm{Syn}_k^{op}$, equipped with the Grothendieck topology where
covers are given by finite families of syntomic maps that are jointly faithfully
flat, is called the {\bf syntomic site of $k$}. An algebraic stack
$\Xscr/k$ is called {\bf syntomic} if there exists a syntomic cover $U \to \Xscr$ with $U$ a syntomic $k$-scheme.
\end{notation}

\begin{example}
\label{synstacks}
Say $G/k$ is a flat and finitely presented affine group scheme. It is known that
$G$ is a syntomic $k$-scheme. It follows that $B G$ is a syntomic $k$-stack:
the canonical map $\mathrm{Spec}(k) \to BG$ is a syntomic cover as it is a
$G$-torsor. In fact, $BG$ is actually a smooth $k$-stack: it suffices to check
this fiberwise, and there it follows by realizing $G$ as a closed subgroup
scheme $G \hookrightarrow \mathrm{GL}_n$ and noting that the resulting map
$\mathrm{GL}_n/G \to BG$ is a smooth surjection with a smooth source.
Nevertheless, it is often more convenient in calculations to work with the
syntomic cover $\mathrm{Spec}(k) \to BG$ (which is functorially defined in
the group scheme $G$) rather than some non-canonical smooth atlas for $BG$.
\end{example}

\newcommand{\SCR}{\mathrm{sCAlg}}
\newcommand{\sCAlg}{\mathrm{sCAlg}}

Our goal is to give a definition of Hochschild and derived de Rham cohomology
(as well as variants) for syntomic $k$-stacks. Let us first recall the
definitions of the relevant functors  in the affine case, i.e., as functors on
$\mathrm{Syn}_k^{op}$; we shall later extend these to all syntomic $k$-stacks
via descent.

\begin{definition}
\label{def:invariants}
Fix $R \in \mathrm{Syn}_k$. 
\begin{enumerate}[(a)]
\item For $i \geq 0$, we write $\bigwedge^i L_{R/k} \in D(R)$ for the $i$th derived wedge
power (in $R$-modules) of the cotangent complex $L_{R/k}$; these together form
the \textbf{Hodge cohomology} of $R$. 
    \item We let $\HH(R/k) = R \otimes^L_{R \otimes_k R } R \in D(k)$ denote the  {\bf Hochschild homology} of $R$ relative to $k$. 

 \item 
The object $\HH(R/k)$ is equipped with  a $k$-linear $S^1$-action, and we let 
$\HC^-(R/k) = \HH(R/k)^{hS^1} \in D(k)$ denote the \textbf{negative cyclic homology} and 
$\HP(R/k) = \HH(R/k)^{tS^1} \in D(k)$ denote the \textbf{periodic cyclic homology.}

\item We let $\THH(R) \in \Sp$ denote the \textbf{topological Hochschild
homology} of
$R$, which is a spectrum (even an $E_\infty$-ring spectrum) with an
$S^1$-action; we write $\TP(R) = \THH(R)^{tS^1}$ for its \textbf{topological periodic
cyclic homology}. See \cite{nikolaus-scholze} for a modern account of
topological Hochschild homology and the structure on it.  
Recall also that
if $k$ is a perfect field of characteristic $p$, then 
$\TP(R)/p \simeq \HP(R/k)$, as in \cite[Theorem 6.7]{BMS2} or \cite[Theorem
3.4]{antieau-mathew-nikolaus}. Moreover, while $\TP(R)$ is typically a
spectrum, if $R$ is an $\FF_p$-algebra, $\TP(R)$ is naturally an
object of $D(\ZZ_p)$.

\item We let $L \Omega_{R/k}$ denote the \textbf{derived de Rham complex}
of $R$, equipped with the \textbf{derived Hodge filtration} $\Fil_H^\star
L\Omega_{R/k}=L
\Omega^{\geq\star}_{R/k}$. Then $L\Omega^{\geq
\star}_{R/k}$ is naturally an object of the filtered derived category $DF(k)$
(see \cite[\S 5]{BMS2} for more on the filtered derived category). By
definition, if $R$ is a finitely generated polynomial ring over $k$, then
$L\Omega^{\geq \star}_{R/k} \simeq \Omega^{\geq \star}_{R/k}$, i.e., the derived de Rham complex with the derived
Hodge filtration agrees with the ordinary de Rham complex and the
filtration b\^{e}te. We then define
$L\Omega^{\geq \star}_{R/k}$ for general $R$ via left Kan
extension. When $k$ has characteristic $p> 0$ (which will be the case in our
applications), it follows from \cite{bhatt-padic} that $L\Omega^{\geq \star}_{R/k} \simeq \Omega^{\geq
\star}_{R/k}$ for smooth $k$-algebras; in particular, $L\Omega_{R/k}$ is
complete for the Hodge filtration for $R$ smooth. 

\begin{remark}
\label{dRcrys}
Assume $k$ has characteristic $p$ and is perfect. It was shown in
\cite[Theorem~3.27]{bhatt-padic} that derived de Rham cohomology of syntomic
algebras can be computed via crystalline cohomology, i.e., for any syntomic
$k$-algebra $R$, there is a natural isomorphism $L\Omega_{R/k} \simeq
R\Gamma_{\crys}(\mathrm{Spec}(R)/k)$, with the Hodge filtration on
$L\Omega_{R/k}$ matching up with the filtration coming from divided powers of
the ideal sheaf on $R\Gamma_{\crys}(\mathrm{Spec}(R)/k)$. Thus, this invariant
admits a ``non-derived'' definition. The derived definition is nevertheless
useful as it is often easy to compute the cotangent complex and its derived
exterior powers (especially once we extend to stacks).
\end{remark}

\item Suppose $k$ is a perfect ring of characteristic $p$. In this situation,
    we write $R\Gamma_{\crys}(\mathrm{Spec}(R)) \in D(W(k))$ for the {\bf
    crystalline cohomology} of $R$; as $k$ is perfect, we can take this to mean
    either absolute crystalline cohomology relative to the pd-base
    $(\mathbf{Z}_p,(p))$ or crystalline cohomology relative to the pd-base
    $(W(k),(p))$ without changing its meaning.  By generalities on crystalline
    cohomology and Remark~\ref{dRcrys}, we can regard
    $R\Gamma_{\crys}(\mathrm{Spec}(R)) \in D(W(k))$ as a lift of $L\Omega_{R/k}
    \in D(k)$. We refer to \cite[\S 8]{BMS2} for a further discussion of this
    theory, including a description via the derived de Rham--Witt complex.

\end{enumerate}
\end{definition} 

In order to extend these functors to syntomic $k$-stacks, we need the following descent result:

\newcommand{\triplearrows}{\begin{smallmatrix} \to \\ \to \\ 
\to \end{smallmatrix} }

\begin{theorem}
\label{BMS2descent}
The following assignments give sheaves on $\mathrm{Syn}_k^{op}$:
\begin{enumerate}
    \item[{\rm (1)}]  the $D(k)$-valued functors $R \mapsto \bigwedge^i
        L_{R/k}$ (for all $i \geq 0$), $\HH(R/k), \HC^-(R/k), \HP(R/k)$;
\item[{\rm (2)}]  the $\mathrm{Sp}$-valued functors $R \mapsto \THH(R), \TP(R)$;
\item[{\rm (3)}] the $D(k)$-valued functor $R
    \mapsto {L\Omega}_{R/k}$ when $k$ has characteristic $p$;
\item[{\rm (4)}] the  $D(W(k))$-valued functor $R \mapsto
    R\Gamma_\crys(\mathrm{Spec}(R))$ when $k$ is perfect of characteristic $p$. 
\end{enumerate}
   In other words, given a syntomic cover $R \to R'$, the natural
    maps
    \[ F(R) \to \mathrm{Tot}\left( \cosimp{F(R')}{F(R' \otimes_R R')}{F(R' \otimes_R R' \otimes_R R')} \right) \]
 are equivalences for $F$ any of the above functors on $\mathrm{Syn}_k$.
\end{theorem} 
\begin{proof}
We refer to \cite[Remark 2.8]{bhatt-completions} for the cotangent complex. In
fact, \cite[\S 3]{BMS2} covers all the functors in (1) and (2), while
\cite[Example 5.12]{BMS2} covers (3); the claim in (4) follows formally from
that in (3) since $R\Gamma_\crys(\Spec(R))$ is derived $p$-complete 
(see \cite[Tag 091N]{stacks-project} for a treatment of this notion)
and $R\Gamma_\crys(\mathrm{Spec}(R))/p \simeq L\Omega_{R/k}$.
\end{proof}

\begin{remark}
    The results of \cite{BMS2} are more general in that they show descent for
    stronger Grothendieck topologies (such as the quasisyntomic topology for all
    classes of functors above, and even the flat topology for the first two).
    These stronger results are critical to the methods of \cite{BMS2}. However, for
    the purpose of geometric applications in this paper, the preceding generality
    suffices.
\end{remark}

\begin{construction}[Cohomology of stacks]
\label{cons:invariantsstacks}
Fix a sheaf $F$ of spectra on $\mathrm{Syn}^{op}_k$. As $F$ is a Zariski sheaf, we know how to make sense of $F(X)$ for any syntomic $k$-scheme $X$ by Zariski descent: we set $F(X) = R\Gamma(X_{Zar}^{aff}, F)$, where $X_{Zar}^{aff}$ denotes the category of affine opens in $X$ equipped with the usual topology. Similarly, if $\Xscr$ is a syntomic $k$-stack, there is a tautological way to make sense of $F(\Xscr)$ by syntomic descent, i.e., we set 
\[ F(\Xscr) := R\Gamma(\mathrm{Syn}^{op}_{k,/\Xscr}, F).\] 
In fact, we can be more explicit in practice: if $\Xscr$ is a
quasicompact syntomic $k$-stack with affine diagonal, then there exists an
affine syntomic cover $U \to \Xscr$ with $U$ a syntomic affine $k$-scheme, and
\v{C}ech descent gives an identification
\[ R\Gamma(\mathrm{Syn}^{op}_{k,/\Xscr}, F) \simeq \mathrm{Tot} \left( \cosimp{F(U)}{F(U \times_{\Xscr} U)}{F(U \times_{\Xscr} U \times_{\Xscr} U)} \right), \]
thus allowing one to compute the left side in terms of the value of $F$ on affines; a similar description applies to all syntomic $k$-stacks if one allows $U$ to be a possibly non-affine syntomic $k$-scheme. Applying this construction to the functors from Theorem~\ref{BMS2descent}, we can obtain the following functors on syntomic $k$-stacks:
\begin{enumerate}
    \item[(1)] the $D(k)$-valued functors $\Xscr \mapsto R\Gamma(\Xscr,
        \bigwedge^i L_{\Xscr/k}) \ (\forall i), \HH(\Xscr/k), \HP(\Xscr/k)$;
    \item[(2)] the $\mathrm{Sp}$-valued functors $\Xscr \mapsto \THH(\Xscr), \TP(\Xscr)$;
    \item[(3)]  the $D(k)$-valued functor $\Xscr
\mapsto R\Gamma_{\mathrm{dR}}(\Xscr/k)$ when $k$ has characteristic $p$;
\item[(4)] the $D(W(k))$-valued functor $\Xscr \mapsto R\Gamma_\crys(\Xscr)$ when
    $k$ is perfect of characteristic $p$.
\end{enumerate}
\end{construction}

The construction as a totalization above also immediately makes it clear that if
a sheaf $F$ has certain structural features when evaluated on syntomic $k$-schemes (resp. smooth $k$-schemes), it does so on syntomic $k$-stacks (resp. smooth $k$-stacks) as well. We shall implicitly exploit this observation later when extending certain natural filtrations on the invariants from Definition~\ref{def:invariants} to the stacky setting.
Nonetheless, there are some subtleties.  

\begin{remark}[Comparison with the stack-theoretic cotangent complex]
For a syntomic $k$-stack $\Xscr$, one can show that $R \Gamma( \Xscr,
L_{\Xscr/k})$ as defined above is the global sections of the stack-theoretic
cotangent complex $L_{\Xscr/k}$, considered as a quasicoherent complex on
$\Xscr$; this is the reason for the above notation. This follows 
from the transitivity triangle, which shows that the global sections 
of $L_{\mathfrak{X}/k}$ satisfies syntomic descent in $\mathfrak{X}$ (as in
\cite[Remark 2.8]{bhatt-completions}). 
\end{remark}

\begin{warning}[Comparison with Hochschild homology of perfect complexes]
Fix a syntomic $k$-stack $\Xscr$. The object $\HH(\Xscr/k)$ constructed above
does not (in general) coincide with $\HH(\Perf(\Xscr)/k)$, the Hochschild
homology of the $k$-linear stable $\infty$-category of perfect complexes on
$\Xscr$: there is always a natural map
$\HH(\Perf(\Xscr)/k)\rightarrow\HH(\Xscr/k)$, but it will not be an
equivalence in general. For example, for $\Xscr := \B\Gm$, one finds that both
sides are concentrated in degree $0$ where we obtain the completion map
$k[t^{\pm 1}]\rightarrow k\llbracket t-1\rrbracket$ (cf. \Cref{BGmHdR}). 
 \end{warning}

\begin{warning}[$\HP$ versus the Tate construction]
For any syntomic $k$-stack $\Xscr$, the object $\HH(\Xscr/k)$
inherits an $S^1$-action. 
It is always true that $\HC^-( \Xscr/k) = \HH(\Xscr/k)^{hS^1}$
since we can commute limits. However, $\HP( \Xscr/k)$
may differ from $\HH(\Xscr/k)^{tS^1}$ since the $S^1$-Tate construction
does not generally commute with limits (compare \Cref{twodifferentHP} below). 
\end{warning}

\newcommand{\prestk}{\mathrm{PreStk}}

\section{Spectral sequences for stacks}
\label{sec:ss}

Continuing the notation of \S \ref{sec:invariants}, we explain how the invariants introduced
in Construction~\ref{cons:invariantsstacks}  come equipped with certain natural filtrations
leading to spectral sequences. The differentials $d_r$ in our spectral sequences have bidegree
$(r,1-r)$.

\begin{definition}
\label{def:sss}
If a sheaf $F$ of spectra on $\mathrm{Syn}_k^{op}$ is equipped with a complete
descending $\mathbf{N}$-indexed filtration by sheaves, then its value on a
syntomic $k$-stack $\Xscr$ also admits a similar filtration. Applying this observation allows us to construct the following spectral sequences.
\begin{enumerate}
\item[(a)] 
Recall that for $R \in \mathrm{Syn}_k$, we have a complete descending $\NN$-indexed
multiplicative $S^1$-equivariant {\bf HKR filtration} $\Fil^\star_{\HKR}\HH(R/k)$ with graded
pieces
$\gr^t_\HKR\HH(R/k)\we\wedge^t L_{R/k}[t]$, obtained by left Kan extending
the Postnikov filtration on polynomial algebras.
The HKR filtration on $\HH(-/k)$ induces a complete descending $\mathbf{N}$-indexed
filtration of $\HH(\Xscr/k)$ with $\mathrm{gr}^t$ given by $R\Gamma(\Xscr, \wedge^t
L_{\Xscr/k})[t]$. In particular, for any syntomic $k$-stack $\Xscr$, we obtain the
{\bf HKR spectral sequence}
\[ E_2^{s,t}=H^s(\Xscr,\bigwedge^{-t} L_{\Xscr/k})\Rightarrow H^{s+t}(\HH(\Xscr/k)).\]
The HKR spectral sequence degenerates in characteristic zero
by~\cite{toen-vezzosi-simpliciales}.
\item[(b)] Assume $\Xscr$ is a smooth $k$-stack where $k$ has characteristic
    $p>0$. Restricting attention to smooth
$k$-algebras and applying the reasoning used above, in conjunction with the last
sentence of Definition~\ref{def:invariants}(e), shows that the de Rham
cohomology $R\Gamma_{\mathrm{dR}}(\Xscr/k)$ admits a complete descending $\mathbf{N}$-indexed filtration
$\mathrm{Fil}^{\star}_H R \Gamma_{\mathrm{dR}}( \Xscr/k)$
with $\mathrm{gr}^i$ given by $R\Gamma(\Xscr, \wedge^i L_{\Xscr/k})[-i]$. In particular, we obtain the {\bf Hodge--de Rham spectral sequence}
$$E_1^{s,t}=H^t(\Xscr,\bigwedge^s L_{\Xscr/k})\Rightarrow H^{s+t}_{\mathrm{dR}}(\Xscr/k).$$

\item[(b')] Assume $k$ is perfect of characteristic $p$. For any $k$-algebra $R$, the object
$L\Omega_{R/k}$ comes endowed with a functorial increasing exhaustive $\mathbf{N}$-indexed
filtration, called the conjugate filtration, with $\mathrm{gr}_i$ given by $\bigwedge^i
L_{R^{(1)}/k}[-i]$ (see \cite{bhatt-padic}). If one restricts attention to syntomic $k$-algebras, these graded pieces
are coconnective. As totalizations of cosimplicial coconnective objects commute with
filtered colimits, we learn by descent that for any syntomic stack $\mathcal{X}/k$, we have
a functorial increasing exhaustive $\mathbf{N}$-indexed filtration on
$R\Gamma_{\dR}(\mathcal{X}/k)$ with $\mathrm{gr}_i$ given by $R\Gamma(\mathcal{X}, \bigwedge^i
L_{\mathcal{X}^{(1)}/k})[-i]$. In particular, we obtain the {\bf conjugate spectral
sequence}
\[  E_2^{s,t} = H^s(\mathcal{X}, \bigwedge^t L_{\mathcal{X}^{(1)}/k}) \Rightarrow
H^{s+t}_{\dR}(\mathcal{X}/k).\]
\end{enumerate}

\begin{remark}
Comparing the $E_2$-terms of the conjugate spectral sequence with the $E_1$-terms of the
Hodge--de Rham spectral sequence shows the following: if both the Hodge and de Rham
cohomology groups of $\mathcal{X}/k$ are finite dimensional in each degree, then the
Hodge--de Rham spectral sequence degenerates if and only if the conjugate spectral sequence
degenerates.
\end{remark}

The preceding discussion also extends to $\mathbf{Z}$-indexed filtrations provided the
graded pieces become highly coconnective for $i \to -\infty$. By the main results of
\cite{BMS2}, this yields the following two spectral sequences.
\begin{enumerate}
\item[(c)] Assume $k$ has characteristic $p> 0$ and $\Xscr$ is a smooth $k$-stack. The motivic filtration on $\HP(-/k)$
constructed in \cite{BMS2} (in the $p$-complete setting) and in
general in \cite{antieau-derham}
induces a complete exhaustive descending $\mathbf{Z}$-indexed filtration on $\HP(\Xscr/k)$
with $\mathrm{gr}^i$ given by $R\Gamma_{\mathrm{dR}}(\Xscr/k)[2i]$. In particular, we obtain
the {\bf de Rham--$\HP$ spectral sequence}
\[ E_2^{s,t}=H^{s-t}_{\mathrm{dR}}(\Xscr/k) \Rightarrow H^{s+t}(\HP(\Xscr/k)).\]
There is a variant for $\HC^-(\Xscr/k)$: one has a 
complete exhaustive $\mathbf{Z}$-indexed descending 
filtration on 
$\HC^-(\Xscr/k)$ with $\gr^t$ given by 
$\mathrm{Fil}^t_H R \Gamma_{\mathrm{dR}}(\Xscr/k)[2t]$, and a similar spectral sequence. 

\item[(d)] Assume $k$ is a perfect ring of characteristic $p$ and
$\Xscr$ is a smooth $k$-stack. The motivic filtration on $\TP(-)$ (cf.
\cite{BMS2}) 
induces a complete exhaustive descending $\mathbf{Z}$-indexed filtration on $\TP(\Xscr)$
with $\mathrm{gr}^i$ given by $R\Gamma_\crys(\Xscr)[2i]$. In particular, we obtain the
crystalline version of the de Rham--$\HP$ spectral sequence, namely the {\bf
    crystalline--$\TP$ spectral sequence}
\[ E_2^{s,t}=H_\crys^{s-t}(\Xscr)\Rightarrow H^{s+t}(\TP(\Xscr;\ZZ_p));\]
the target is the $p$-completion of topological periodic cyclic homology. 

\item[(e)] For any syntomic $k$-stack $\Xscr$, we have the {\bf Tate spectral sequence}
	\[E_2^{s,t}=H^s_{\mathrm{Tate}}(BS^1,H^t(\HH(\Xscr/k)))\Rightarrow
    H^{s+t}((\HH(\Xscr/k))^{tS^1}).\]
    If $\Xscr$ is a syntomic $k$-scheme, 
    then we can identify 
    $\HH(\Xscr/k)^{tS^1}$ with $\HP( \Xscr/k)$. 
\end{enumerate}

\end{definition}

\newcommand{\qcoh}{\mathrm{QCoh}}
\newcommand{\md}{\mathrm{Mod}}
\renewcommand{\hom}{\mathrm{Hom}}

\begin{remark}
The smoothness assumptions were made in Definition~\ref{def:sss} to ensure that derived de
Rham cohomology is complete for the Hodge filtration. We could drop this assumption entirely
if we replace derived de Rham cohomology with its Hodge-completed variant (and derived
        crystalline cohomology with its Nygaard completed variant). However, since the
stacks that we shall encounter later are smooth, we prefer to stick to the limited
generality introduced above.
\end{remark}

We organize the spectral sequences introduced above in Figure~\ref{fig:quartet1}, borrowed from~\cite{antieau-bragg}.
\begin{figure}[H]
  \centering
  \begin{tikzcd}
    & H^*(\HH(\Xscr/k))\arrow[Rightarrow]{dr}{\text{Tate}}&\\
      H^*(\Xscr,\bigwedge^*L_{\Xscr/k})\arrow[Rightarrow]{ur}{\text{HKR}}\arrow[Rightarrow]{dr}[swap]{\text{Hodge--de
      Rham}}&&H^*(\HP(\Xscr/k))\\
          &H^*_{\mathrm{dR}}(\Xscr/k) \arrow[Rightarrow]{ur}[swap]{\text{de Rham--HP}}&
  \end{tikzcd}
  \caption{The Hodge quartet.}
  \label{fig:quartet1}
\end{figure}

\begin{remark}
    The Hodge--de Rham spectral sequence degenerates in characteristic zero for
	 smooth proper schemes by Hodge theory. Moreover, if $k$ is a perfect field
	 of characteristic $p$, and $X/k$ is a smooth proper $k$-scheme with
	 $\dim(X)\leq p$ that lifts to $W_2(k)$, then the Hodge--de Rham spectral
	 sequence degenerates by Deligne--Illusie~\cite{deligne-illusie}. Remarkably,
	 it is still unknown whether the hypothesis on dimension is necessary in the
	 preceding statement: could it be true that the Hodge--de Rham spectral
	 sequence degenerates for any smooth proper scheme over a perfect field $k$ of
	 characteristic $p$ which is liftable to $W_2(k)$ (or even $W(k)$)? This question was
     explicitly raised in \cite[Problem~7.10]{illusie-frobenius}, and Deligne--Illusie
     presumed that the answer is `no'.
\end{remark}

\begin{remark}[The non-commutative Tate spectral sequence]
In Definition~\ref{def:sss}, the first four spectral sequences crucially use
algebraic geometry. However, the Tate spectral sequence extends to the non-commutative
setting: for any $k$-linear stable $\infty$-category $\mathcal{C}$, there is 
a spectral sequence
\[E_2^{s,t}=H^s_{\mathrm{Tate}}(BS^1,H^t(\HH(\Cscr/k)))\Rightarrow H^{s+t}(\HP(\Cscr/k))\]
In this context, the Tate spectral sequence is also called the noncommutative
Hodge--de Rham spectral sequence. Let $k$ be a perfect field of characteristic
$p$. A result of Kaledin~\cite{kaledin-nhdr,kaledin-spectral} (see
also~\cite{mathew-degeneration}) implies that if $\Cscr$ is a smooth proper
$k$-linear stable $\infty$-category (such as $\Perf(X)$ where $X$ is a smooth proper
        $k$-scheme) such that $H^i(\HH(\Cscr/k))=0$ for $i\notin[-p,p]$ (the noncommutative
            analogue of $\dim(X)\leq p$) and if $\Cscr$ lifts to $W_2(k)$, then the Tate
        spectral sequence degenerates at $E_2$. Kaledin used this to prove that the Tate
        spectral sequence degenerates for smooth proper dg categories over characteristic
        zero fields, which together with HKR in characteristic zero implies Hodge--de Rham
        degeneration.
\end{remark}

\begin{remark}[A degeneration criterion]
    Suppose that $X$ is a smooth and proper variety over a perfect field $k$
	 of characteristic $p$. In this case, all of the
    $k$-vector spaces appearing in Figure~\ref{fig:quartet1} are
    finite dimensional. Thus, we can use dimension counts to make conclusions
    about degeneration of the spectral sequences. For example, if $\dim X\leq
    p$, then the HKR spectral sequence degenerates \cite{antieau-vezzosi}. This implies that
    Tate spectral sequence degenerates if and only if both the Hodge--de Rham and
    the de Rham--$\HP$ spectral sequences degenerate. In particular, if the
    Hodge--de Rham spectral sequence does not degenerate, then neither does the
    Tate spectral sequence. 
\end{remark}

The literature (see for example~\cite[Remarques~2.6(i)]{deligne-illusie}) provides many examples of smooth
proper surfaces where the Hodge--de Rham spectral sequence, and hence the
Tate spectral sequence, does not degenerate. In this paper, we provide examples that witness the
non-degeneration of the remaining spectral sequences, i.e., the HKR, de Rham--$\HP$, and
crystalline--$\TP$ spectral sequences.

\section{Classifying space counterexamples}

In  this section, we establish counterexamples to degeneration of the HKR
spectral sequence for certain algebraic stacks. Our examples are the
classifying stacks of finite flat group schemes. In fact, the group schemes
$\mu_p$, $\mu_p \times \mu_p$, and $\alpha_p$ already lead to the desired
counterexamples.\footnote{These examples exhibit qualitatively different behaviour: the group schemes $\mu_p$ or $\mu_p \times \mu_p$ lift to $W_2(k)$ compatibly with Frobenius and their Hodge-de Rham spectral sequences always degenerate (Proposition~\ref{Bmuss}), while the group scheme $\alpha_p$ does not lift to $W_2(k)$ and its Hodge-de Rham spectral sequence does not degenerate (Proposition~\ref{dRofHP} or Remark~\ref{ConjBalphap}). Nevertheless, either example leads to a non-degenerate HKR spectral sequence.} Later, we will approximate, in the sense of
Totaro~\cite{totaro-chow}, these classifying stacks by smooth projective
$k$-schemes to prove Theorem~\ref{thm:intro1}.

In this section, we work over a fixed perfect field $k$ of characteristic
$p>0$.  For a finite flat $k$-group scheme $G$, we will let $BG$ denote
the classifying stack of $G$-torsors for the $\fppf$-topology. Note that these stacks are always smooth (Example~\ref{synstacks}).

Our arguments begin with the calculation of the Hodge cohomology of $BG$, and are phrased in terms of the 
co-Lie complex 
$\mathrm{coLie}(G) \in D(BG)$ of \cite[Ch. VII, 3.1.2]{illusie-cotangent2}. This
is a $G$-equivariant refinement of $e^* L_{G/k}$ for $e\colon \Spec(k)\to G$ the
identity section, 
and can be identified the cotangent complex of
the stack $BG$ up to a shift. 
When $G$ is smooth, $\mathrm{coLie}( G)$ is the linear dual of the adjoint 
representation of $G$.  
The main technical tool is the following result.

\begin{theorem}[{Totaro, cf. \cite[Theorem 3.1]{totaro}}] 
There is a multiplicative, graded isomorphism, 
\begin{equation} \label{BGhodgecoh} R\Gamma\left( BG, \bigoplus_{i \geq 0} \bigwedge^i L_{BG/k} \right)
\simeq R\Gamma\left( G, \bigoplus_{i \geq 0}
\mathrm{Sym}^i(\mathrm{coLie}(G))[-i]\right),\end{equation}
in $D(k)$,
 where the right side denotes ``rational cohomology'' of $G$-representations. 
\end{theorem}

The result comes from the fact (which to formulate we use the cotangent
complex of stacks) that if $\pi\colon\mathrm{Spec}(k) \to B G$ denotes the tautological map, then there is a multiplicative, graded, $G$-equivariant isomorphism 
\[ \pi^* \left(\bigoplus_{i \geq 0} \bigwedge^i L_{B G/k}\right) \simeq
\bigoplus_{i \geq 0} \mathrm{Sym}^i(\mathrm{coLie}(G))[-i],\]
which implies
the isomorphism \eqref{BGhodgecoh}
in $D(k)$.

In our examples, the basic source for the non-degeneration is the following. 
The Frobenius on the classifying stacks $B \mu_p, B \alpha_p$ factors through a
point, which forces the Frobenius to act by zero on the zeroth Hochschild
homology. However, the $p$th power map is not zero in the relevant Hodge cohomology
(in fact, the $p$th power map does not agree with the map induced by
the Frobenius), which forces
the existence of differentials in the HKR spectral sequence. 
Another phenomenon that leads to non-degeneration of the Hodge-to-de Rham
spectral sequence for $B \alpha_p$ is 
its failure to lift to $W_2(k)$, which leads to the non-degeneration of the
conjugate spectral sequence. 

\subsection{$B\mu_p$}\label{sec:bmup}

In this section, we analyze the spectral sequences for $B\mu_p$. Let us begin
with a (special case of much more general) degeneration criterion for the
Hodge--de Rham spectral sequence.

\begin{lemma}
\label{BmuHdRLift}
Let $\Xscr = BG$ for a diagonalizable group scheme $G$ (such as $\mu_p$
or $\mathbb{G}_m$). Then the conjugate and Hodge--de Rham spectral sequences
for $\Xscr$ degenerate.
\end{lemma}
\begin{proof}
    Using \cite[Remark 2.2 (ii)]{deligne-illusie} as well as functorial
    simplicial resolutions, one shows that for any syntomic $k$-algebra $R$
    equipped with a lift $\tilde{R}$ to $W_2(k)$ together with a lift
    $\tilde{\phi}\colon \tilde{R} \to \tilde{R}$ of the Frobenius, there is a
    functorial (in the pair $(\tilde{R},\tilde{\phi})$)  isomorphism $\oplus_{i
    \geq 0} \wedge^i L_{R^{(1)}/k}[-i] \simeq L\Omega_{R/k}$, i.e., the
    conjugate filtration splits functorially in the lifting data. Now the
    syntomic hypercover of $\Xscr = BG$ given by the \v{C}ech nerve of
    $\mathrm{Spec}(k) \to BG$ is a simplicial syntomic affine $k$-scheme that
    comes equipped with such lifting data compatibly with the simplicial
    structure maps. Applying the preceding splitting levelwise and totalizing
    gives an isomorphism $R\Gamma_{\mathrm{dR}}(BG/k) \simeq \oplus_{i \geq 0}
    R\Gamma(BG, \wedge^i L_{BG/k}[-i])$, i.e., the conjugate filtration
    splits, and thus the conjugate spectral sequence degenerates. Dimension
    considerations now show that the Hodge--de Rham spectral sequence must also
    degenerate.
\end{proof}

This allows us to recover the following standard calculation. As usual, let
$P(c)$ denote a polynomial ring over $k$ on the generator $c$ and let $E(d)$
denote an exterior algebra over $k$ on the generator $d$.

\begin{example}[Cohomology of $B \mathbb{G}_m$]
\label{BGmHdR}
We have that $\mathrm{coLie}( \mathbb{G}_m) $ is the trivial representation of
$\mathbb{G}_m$. 
As $R\Gamma(B \mathbb{G}_m, \mathcal{O}) \simeq k$, it follows that
$H^*(B\mathbb{G}_m, \wedge^{\ast} L_{B \mathbb{G}_m/k}) \simeq P(c)$, where
$c$ has bidegree $(1, 1)$. By Lemma~\ref{BmuHdRLift}, we also have
$H^*_{\mathrm{dR}}(B\mathbb{G}_m) \simeq P(c)$. The generator $c$ can be
explicitly chosen as the first Chern class of the tautological line bundle on
$B\mathbb{G}_m$.
The HKR spectral sequence for $H^*(\HH( B\mathbb{G}_m/k))$ degenerates and we find that
$H^*(\HH(B \mathbb{G}_m/k))\simeq k[[u]]$ for a class $u$ in degree zero. 
In fact, all the spectral sequences considered above degenerate for $B
\mathbb{G}_m$ as there is no room for differentials. 
\end{example}

Let us first review the Hodge and de Rham cohomologies of $B\mu_p$. 
Part (i) of the following proposition is a consequence of the fact that
$\mathrm{coLie}(\mu_p) \simeq \mathcal{O} \oplus \mathcal{O}[1] \in D(\mu_p)$, which 
yields the calculation of Hodge cohomology. Part (ii) follows from the fact that
the group scheme $\mu_p$ lifts to characteristic zero with a lift of Frobenius.

\begin{proposition}[{Cf. \cite[Proposition~10.1]{totaro}}]
\label{HdRBmup}
    \begin{enumerate}
        \item[{\rm (i)}] The Hodge cohomology of $B\mu_p$ is given by
            $$\H^*(B\mu_p,\bigwedge^* L_{B\mu_p/k})\iso E(d)\otimes P(c),$$
            where $d \in H^0(\B\mu_p, L_{B\mu_p/k})$ and $c \in H^1(B\mu_p, L_{B\mu_p/k})$.
        \item[{\rm (ii)}] The Hodge--de Rham and conjugate spectral sequences degenerate for
            $B\mu_p$ and we have an isomorphism
            $$H^*_{\mathrm{dR}}(B\mu_p/k)\iso E(d)\otimes P(c),$$ where
            $|d|=1$ and $|c|=2$.
    \end{enumerate}
\end{proposition}

\begin{proposition}
\label{HHBmup}
The Hochschild homology ring $H^*(\HH(B\mu_p/k))$ is isomorphic to $k[c]/(c^p)$ where $|c|=0$. In particular, it is concentrated in degree $0$.
\end{proposition}

\begin{proof}
    We use the HKR spectral sequence to calculate $H^*(\HH(B\mu_p/k))$. Its $E_2$-page is
 calculated using Proposition~\ref{HdRBmup}.  Using the notation there, one
 sees that for degree reasons, the class $c\in\H^1(B\mu_p,L_{\B\mu_p/k})$
 must be permanent. Thus, it defines a nonzero class of $H^0(\HH(B\mu_p/k))$
 that we also call $c$. Note that $c \in H^0(\HH(B\mu_p/k))$ is annihilated by
 the pullback along the tautological map $\mathrm{Spec}(k) \to B\mu_p$. On the
 other hand, the Frobenius $\varphi$ on $B\mu_p$ factors through this map, so
 we see that $\varphi(c)=0$ in $H^0(\HH(B\mu_p/k))$. But $\varphi$ acts on
 $H^0(\HH(B\mu_p/k))$ by the $p$-power map,\footnote{Given an $\mathbf{F}_p$-algebra $R$, the
 endomorphism of $\HH(R/\mathbf{F}_p)$ induced by the Frobenius on $R$
 coincides with the Frobenius endomorphism of the simplicial commutative
 $\mathbf{F}_p$-algebra $\HH(R/\mathbf{F}_p)$. Applying this observation to a
 hypercover shows that for any algebraic stack $\Xscr/\mathbf{F}_p$, the
 endomorphism of $H^0(\HH(\Xscr/\mathbf{F}_p))$ induced by the Frobenius on
 $\Xscr$ is the $p$-power map.} so multiplicativity of the HKR spectral
 sequence shows that $\varphi(c)$ is represented by
 $c^p\in H^p(B\mu_p,\bigwedge^p L_{B\mu_p/k})$, whence the latter must be a
 boundary in the HKR spectral sequence. To proceed further, observe that we
 have a filtered map $\HH(B\mu_p/k) \to \HH(B\mu_p/k)/\Fil^p_{\HKR}$
 as well as a formality isomorphism $\HH(B\mu_p/k)/\Fil^p_{\HKR} \simeq
 \oplus_{i \leq p-1} R\Gamma(B\mu_p, \wedge^i L_{B\mu_p/k}[i])$ in $DF(k)$. By
 contemplating the induced map on spectral sequences, we learn that the only
 differential that can possibly hit $c^p \in
 H^p(B\mu_p,\bigwedge^p L_{B\mu_p/k})$ in the HKR spectral sequence of
 $B\mu_p$ is $d_p(d)$ (up to a unit). This completes the proof as now on the
 $E_{p+1}$-page of the spectral sequence we are left only with the nonzero
 classes $1,c,\ldots,c^{p-1}$.
\end{proof}

\begin{theorem} 
\label{Bmuss}
For $\B\mu_p$, the following assertions hold.
\begin{enumerate}
    \item[{\rm (1)}] The HKR spectral sequence does not degenerate.
There is a nonzero differential 
$d_{p}\colon H^0( B\mu_p, L_{B \mu_p/k}) \to H^{p}( B \mu_p, \bigwedge^p L_{B
\mu_p/k})  $. 
\item[{\rm (2)}] The Hodge--de Rham and conjugate spectral sequences degenerate.
\item[{\rm (3)}] The Tate spectral sequence for $\HH(B\mu_p/k)^{tS^1}$ degenerates.
\item[{\rm (4)}] The de Rham--$\HP$ spectral sequence does not degenerate. There is a
nonzero differential 
$d_{p}\colon H^1_{\mathrm{dR}}( B\mu_p/k) \to H^{2p}_{\mathrm{dR}}( B \mu_p/k)$. 
\item[{\rm (5)}] The crystalline--$\TP$ spectral sequence degenerates, but the
    resulting filtration on $H^*(\TP(\B\mu_p))$ is not split.
\end{enumerate}
\end{theorem}
\begin{proof}
\begin{enumerate}
    \item[(1)] This was shown in the course of the proof of Proposition~\ref{HHBmup}.
    \item[(2)] This was shown in Proposition~\ref{HdRBmup} (see also Lemma~\ref{BmuHdRLift}).
    \item[(3)] This follows for degree reasons as $\HH(\B\mu_p/k)$ is concentrated in degree $0$.
    \item[(4)] It follows from the calculation of $\HH(B\mu_p/k)$ that 
        $H^*(\HC^-(\B\mu_p/k))\iso P(t)\otimes_k k[c]/(c^p),$ where $|t| = 2.$
In particular, this theory is concentrated in even degrees. Since we know that $\B\mu_p$ has de Rham cohomology in odd degree (Proposition~\ref{HdRBmup}), the
spectral sequence from Hodge-filtered de Rham cohomology to $\HC^-$ cannot
degenerate. 
By naturality, it follows that the de Rham--$\mathrm{HP}$ spectral sequence
cannot degenerate either. Explicitly, we find that the de Rham--$\mathrm{HP}$-spectral sequence has
$E_2$-term given by $E(d) \otimes P(c) \otimes P(t^{\pm 1})$; here $t$ is a
permanent cycle as it comes from the cohomology of $S^1$ and $c$ is a permanent
cycle as it comes from $B \mathbb{G}_m$. 
By the description of $H^*(\HC^-( B \mu_p/k))$, it follows that $c^p = 0$, so we must
have a nonzero differential 
$d_{p}\colon H^1_{\mathrm{dR}}( B \mu_p/k) \to H^{2p}_{\mathrm{dR}}( B \mu_p/k)$ annihilating $c^p$. 
The spectral sequence now shows that 
$H^*(\HP _*(B \mu_p/k))\simeq P(t^{\pm 1}) \otimes k[c]/c^p$, where $|t| = 2$. 

\item[(5)] Since $\TP(\B\mu_p)/p \simeq \HP(B\mu_p/k)$, the calculation in (4)
    implies that $H^*(\TP(\B\mu_p))$ is concentrated in even degrees and
    $p$-torsionfree. On the other hand, we have 
\[H^*_{\crys}(\B\mu_p) \simeq W[c]/(pc),\] 
where $|c|=2$: this follows from the isomorphism $R\Gamma_\crys(\B\mu_p)/p
\simeq R\Gamma_{\dR}(\B\mu_p/k)$, the calculation in Proposition~\ref{HdRBmup},
and the observation that multiplication by $n$ on $H^{2}_\crys(\B\mu_p)$ is
induced by the multiplication by $n$ endomorphism of $\B\mu_p$ (and is thus the
$0$ map when $p \mid n$). It follows that all terms on the $E_2$-page of the
crystalline--$\TP$ spectral sequence are in even degrees, so the spectral
sequence degenerates. As $H^*_{\crys}(B\mu_p)$ contains nonzero $p$-torsion
elements while $H^*(\TP(\B\mu_p))$ is $p$-torsionfree, the filtration on
$H^*(\TP(\B\mu_p))$ coming from this spectral sequence cannot split.
\end{enumerate}
\end{proof}

\begin{remark}[The HKR filtration does not split for Frobenius lifts]
In analogy with Deligne--Illusie \cite[Remark 2.2 (ii)]{deligne-illusie}, one
might wonder the following: given a smooth $k$-algebra with a lift $\tilde{R}$
to $W_2(k)$ and a lift $\tilde{\phi}\colon\tilde{R} \to \tilde{R}$ of the Frobenius,
can one choose an isomorphism $\HH(R/k) \simeq \oplus_i \Omega^i_{R/k}[i]$
splitting the HKR filtration that is functorial in the lifting data
$(\tilde{R},\tilde{\phi})$? Theorem~\ref{Bmuss} (1) shows that this is not
possible (via the argument of Lemma~\ref{BmuHdRLift} to pass from the affine
case to stacks).
\end{remark}

Finally, let us use the calculations above to record an example where the
crystalline--$\TP$ spectral sequence does not degenerate.

\begin{lemma}\label{lem:mupmup}
    The crystalline--$\TP$ spectral sequence for $B(\mu_p\times\mu_p)$ does
    not degenerate.
\end{lemma}

\begin{proof}
We saw in the proof of (3) of \Cref{Bmuss},
 via the de Rham--$\mathrm{HP}$ spectral sequence, that
$H^*(\HP( B \mu_p/k))$ is concentrated in even degrees and given by 
$P(t^{\pm 1}) \otimes_k k[c]/c^p$. 
While $\HP$ does not in general satisfy a K\"unneth formula for syntomic
$k$-stacks, de Rham cohomology
does. Running the de Rham--$\HP$ spectral sequence again, 
we find that $H^*(\HP( B (\mu_p \times \mu_p)/k))$ is concentrated in even degrees. 
 Therefore,
 $H^*(\TP(\B(\mu_p\times\mu_p)))$ is concentrated in even degrees and is
	 $p$-torsion-free, since $\TP/p \simeq \HP$. On the other
    hand, $H^3_\crys(\B(\mu_p\times\mu_p))\iso k$ by K\"unneth. The lemma follows.
\end{proof}

\subsection{$B\alpha_p$}

In this subsection, we calculate everything explicitly for $B\alpha_p$. We shall crucially exploit the natural $\mathbf{G}_m$-action on $B\alpha_p$, induced (ultimately) from the $\mathbf{G}_m$-action on $\mathbf{G}_a$, defined formally as follows.

\begin{observation}[Gradings]
The group scheme $\alpha_p = \mathrm{Spec}(k[t]/(t^p))$ has a natural
$\mathbf{G}_m$-action defined by requiring the function $t$ to have weight $1$. This induces a $\mathbf{G}_m$-action on $B\alpha_p$, and consequently there is a natural weight grading on associated cohomological invariants, such as Hodge, Hochschild, and de
Rham cohomology. Moreover, the differentials in the relevant spectral sequences respect the weight grading.
\end{observation}

\begin{proposition}
\label{HodgeBalphap}
If $p > 2$, 
the Hodge cohomology of $\B\alpha_p$ is given by
\[ H^*(B\alpha_p,\bigwedge^*L_{B\alpha_p/k})\iso E(\alpha)\otimes P(\beta) \otimes E(s) \otimes P(u), \]
where $\alpha \in H^1(B\alpha_p, \mathcal{O})$, $\beta \in H^2(B\alpha_p,
\mathcal{O})$, $s \in H^0(B\alpha_p, L_{B\alpha_p/k})$ and $u \in
H^1(B\alpha_p, L_{B\alpha_p/k})$. Moreover, the weights of $\alpha,\beta,s$
and $u$ are $1,p,p$ and $1$ respectively.
For $p = 2$, we replace $E(\alpha) \otimes P(\beta)$ with $P(\alpha)$. 
\end{proposition}
\begin{proof}
We begin by showing that $H^*(B\alpha_p, \mathcal{O}) \simeq E(\alpha) \otimes P(\beta)$ with degrees and weights as in the Proposition. By Cartier duality\footnote{For a finite $k$-group scheme $G$, the abelian category of coherent sheaves $\mathrm{Coh}(BG)$ on $BG$ can be identified as the category $\mathrm{Rep}^f(G)$ of finite dimensional representations $G$, i.e., as the category $\mathrm{CoMod}^f_{\mathcal{O}(G)}$ of finite dimensional comodules over the $k$-coalgebra $\mathcal{O}(G)$. When $G$ is commutative, this is antiequivalent to the category $\mathrm{Mod}^f_{\mathcal{O}(G^\vee)}$ of finite dimensional modules over $k$-algebra $\mathcal{O}(G)^\vee \simeq \mathcal{O}(G^\vee)$, where $G^\vee$ denotes the Cartier dual of $G$. Under this identification, the trivial representation of $G$ corresponds to the residue field at the origin on $G^\vee$. Computing $\mathrm{Ext}$-groups now gives the isomorphism used above.}, we have $H^*(B\alpha_p, \mathcal{O}) \simeq \mathrm{Ext}^*_{k[s]/(s^p)}(k,k)$, where $k[s]/(s^p)$ denotes the Hopf algebra of functions on the Cartier dual of $\alpha_p$ (and is thus also a copy of $\alpha_p$ itself, but the weight of the generator $s$ is now $-1$). One then calculates using the standard resolution
\[ \left( \cdots  k[s]/(s^p) \xrightarrow{s^{p-1}} k[s]/(s^p) \xrightarrow{s} k[s]/(s^p) \xrightarrow{s^{p-1}} k[s]/(s^p) \xrightarrow{s} k[s]/(s^p) \right) \stackrel{can}{\simeq} k,\]
graded in a natural way, that the answer is as predicted. Alternately, one can find this calculation in \cite[Theorem 2.4]{Friedlander03}.

To compute Hodge cohomology, we first calculate the co-Lie complex. Using the closed immersion $\alpha_p \subset \mathbb{G}_a$ of group schemes, we
learn that $L_{\alpha_p/k}$ is computed by the $2$-term complex $(t^p)/(t^{2p})
\xrightarrow{d} k[t]/(t^p) dt$. 
Restricting along the origin gives a $2$-term
complex of $\alpha_p$-representations computing $\mathrm{coLie}(\alpha_p)$ as
an object of $D(k)$. 
Using this complex, we find that $H^0( \mathrm{coLie}(\alpha_p)) = H^{-1}(
\mathrm{coLie}(\alpha_p)) = k$ are both the trivial one-dimensional representation of $\alpha_p$. 
Furthermore, $H^0$ (corresponding to $dt$)  is concentrated in weight one while
$H^{-1}$ 
(corresponding to $t^p$) is concentrated in weight $p$. 
But then $\mathrm{coLie}(\alpha_p) \in D(B \alpha_p)$ splits as $\mathcal{O} \oplus \mathcal{O}[1]$: the obstruction to splitting is a weight $p-1$ map $\mathcal{O} \to \mathcal{O}[2]$, and there are no such maps by the calculation of $H^2(B\alpha_p, \mathcal{O})$ explained in the previous paragraph. Another approach to seeing this description of $\mathrm{coLie}(\alpha_p)$ is to use that $\alpha_p$ is
commutative and \cite[Ch. VII, Prop. 4.1.1]{illusie-cotangent2}. 

Thus, we learn that, as objects in $D(B\alpha_p)$, we have $\mathrm{coLie}(\alpha_p) \simeq
\mathcal{O} s[1] \oplus \mathcal{O} u$, where $s $ corresponds to a class in $ H^0(B\alpha_p,
L_{B\alpha_p/k})$ which has weight $p$ and $u$ corresponds to a class
$H^1(B\alpha_p,
L_{B\alpha_p/k})$ which has weight $1$. As in Proposition~\ref{HdRBmup}, one then
finds that 
$$
\bigoplus_{i \geq 0} \mathrm{Sym}^i( \mathrm{coLie}(G))[-i]
\simeq E(s) \otimes P(u) \otimes \mathcal{O}  \in D(B\alpha_p).$$
Combining this with the calculation of $H^*(B\alpha_p, \mathcal{O})$ and using the projection formula then gives the desired answer.
\end{proof}

For the next result, we recall that $\HH( B \alpha_p/k)$ acquires an action of
the circle $S^1$, inducing an operator $H^* ( \HH( B \alpha_p/k)) \to H^{*
-1}( \HH( B \alpha_p/k))$ given by multiplication by the fundamental class of
$S^1$; this is also identified (up to 2-periodicity) with
the first differential in the Tate spectral sequence. 

\begin{proposition}\label{HHofBalphap}
    If $p$ is odd, then $H^*(\HH(B\alpha_p/k))$ is isomorphic to $E(\alpha)
    \otimes P(\beta) \otimes k[u]/u^p $ with $\alpha$ having degree $1$ and
    weight $1$, $\beta$ having degree $2$ and weight $p$, and $u$ having degree
    $0$ and weight $1$.  If $p=2$, then $H^*(\HH(B\alpha_p/k))$ is given by
    $k[\alpha]\otimes k[u]/u^p$ if $p=2$ with the same degrees and weights as
    in the odd case.
    The circle action carries $\alpha \mapsto u$ (up to units); in particular, the
    Tate spectral sequence for $\HH(B \alpha_p/k)^{tS^1}$ does not degenerate. 
\end{proposition} 

\begin{proof} 
    We give the proof when $p$ is odd. We will use the HKR spectral sequence
    and the calculation of the $E_2$-page coming from
    Proposition~\ref{HodgeBalphap}. Note that $|\alpha|=(1,0)$,
    $|\beta|=(2,0)$, $|s|=(0,-1)$, and $|u|=(1,-1)$ in the $E_2$-page of the HKR
    spectral sequence.
    
    We begin by noting that $\alpha$ and $\beta$ are permanent cycles arising
    from the map  $R \Gamma(B \alpha_p, \mathcal{O}) \to \HH(B \alpha_p/k)$
    which comes from choosing a basepoint on $S^1$. (In fact, we recall that
    $\HH(\Xscr/k)$ must contain $R \Gamma(\Xscr, \mathcal{O})$ as a summand
    for any syntomic stack $\Xscr/k$.) Next, since every differential respects
    the weight grading, we conclude that $u$ (which has weight $1$) is a
    permanent cycle: all weights that occur on the target of a differential
    emanating from $u$ have weights $> 1$. Finally, we claim $d_{p}(s) = u^p$
    up to units: this is proven like the analogous claim in
    Proposition~\ref{HHBmup}, noting that Frobenius on $B\alpha_p$ factors
    through a point. There are no further differentials (as $\alpha$, $\beta$,
    and $u$ are permanent), so we obtain that $H^*(\HH(B\alpha_p/k))$ has the
    predicted shape.

    For the circle action, we use the following observation: if $R$ is any
    nonnegatively graded commutative
    $k$-algebra with $R_0 = k $, then we have an $S^1$-equivariant equivalence
    in weight $1$,
    $\HH(R/k)_{\mathrm{wt} = 1} \simeq C_*(S^1;k ) \otimes_k (L_{R/k})_{\mathrm{wt} = 1}$
    as one sees by reducing to the free case. 
    In particular, in weight $1$, the circle action on $\HH(-/k)$ is always induced. 
    Since this is functorial, it applies to $B \alpha_p$ too, and we find that 
    $H^*(\HH(B \alpha_p/k))$ in weight $1$ has an induced $S^1$-action, whence the claim. 
\end{proof}

\begin{proposition} 
\label{dRofHP}
For all $p$, the de Rham cohomology of $B\alpha_p$ is given by 
\[ H^*_{\mathrm{dR}}(B\alpha_p/k) \simeq E( \alpha') \otimes P( \beta'),\] 
where $\alpha'$ has degree $1$ and weight $p$, and $\beta'$ has degree $2$ and
weight $p$.  In particular, both the conjugate and the Hodge--de Rham spectral
sequences for $B\alpha_p$ fail to degenerate.
\end{proposition}

\begin{proof}
We use the Hodge--de Rham spectral sequence and the calculation of the
$E_1$-page coming from Proposition~\ref{HodgeBalphap}. Conjugate filtration
considerations show that the abutment can have no terms in weights not
divisible by $p$, so there must be a differential in weight one, which forces
$d_1(\alpha)  = u$ up to units. This forces all the differentials as $\beta, s$
now have to be permanent cycles for weight and degree reasons.
\end{proof}

\begin{remark}[Non-degeneration of the conjugate spectral sequence for $B\alpha_p$]
\label{ConjBalphap}
One can also calculate the de Rham cohomology of $B\alpha_p$ using the
conjugate spectral sequence, giving a different proof of
Proposition~\ref{dRofHP}. As $\alpha_p$ is defined over $\mathbf{F}_p$, we may
assume $k = \mathbf{F}_p$, which allows us to suppress Frobenius twists. The
conjugate spectral sequence takes the form
\[ E_2^{i,j}=H^i(B\alpha_p, \bigwedge^j L_{B\alpha_p/k}) \Rightarrow
H^{i+j}_{\mathrm{dR}}(B\alpha_p/k).\]

The $E_2$-page is again calculated by Proposition~\ref{HodgeBalphap}, except that all weights are multiplied by $p$ (due to the implicit Frobenius twists).

 First, note that $\alpha$ and $\beta$ are permanent cycles as there is no room
 for the differentials. Moreover, by weight considerations, $d_2(u) = 0$, which
 makes $u$ a permanent cycle as the higher differentials have $0$ target. The
 key claim is that $d_2(s) = \beta$ (up to units). Granting this claim, one
 immediately deduces the calculation of $H^*_{\mathrm{dR}}(B\alpha_p/k)$ given
 in Proposition~\ref{dRofHP}, as well as the fact that both the Hodge--de Rham
 and conjugate spectral sequences do not degenerate (by counting dimensions).

To prove the claim $d_2(s) = \beta$ (up to units), we use that $B\alpha_p$ does
{\em not} lift to $\mathbf{Z}/p^2$. This implies that the map
$\mathrm{ob}_{B\alpha_p}\colon L_{B\alpha_p/k} \to \mathcal{O}[2]$ measuring the
failure to lift  to $W_2$ is nonzero. But, for any syntomic stack $\Xscr/k$,
the $d_2$ differential $H^i(\Xscr, L_{\Xscr/k}) \to H^{i+2}(\Xscr,
\mathcal{O})$ in the conjugate spectral sequence is just the map on $H^i$
induced by $\mathrm{ob}_{\Xscr}$ (by \cite[Theorem 3.5]{deligne-illusie}
extended to stacks). Thus, it is enough to show that the map
$H^0(B\alpha_p,\mathrm{ob}_{B\alpha_p})$ is nonzero. For this, write
$L_{B\alpha_p/k} \simeq \mathcal{O}s \oplus \mathcal{O}u[-1]$ using the
generators found in Proposition~\ref{HodgeBalphap}. We must show that the
restriction of $\mathrm{ob}_{B\alpha_p}$ to the first factor $\mathcal{O}s$ is
nonzero. But the restriction of $\mathrm{ob}_{B\alpha_p}$ to the second factor
$\mathcal{O}u[-1]$ is $0$ by comparison with the analogous situation for the
liftable stack $B\mathbb{G}_a$. As $\mathrm{ob}_{B\alpha_p}$ was already
shown to be nonzero, the claim follows.
\end{remark} 

\begin{remark}
Combining Proposition~\ref{dRofHP} (or Remark~\ref{ConjBalphap}) with the approximation result in Theorem~\ref{thm:approximationintro} gives a large supply of examples of smooth projective surfaces in characteristic $p$ where both the Hodge-de Rham and conjugate spectral sequences fail to degenerate. 
\end{remark}

\begin{proposition}
\label{dRHpBalphap}
For all $p$, the de Rham--$\mathrm{HP}$ spectral sequence for $B \alpha_p$
degenerates. 
\end{proposition} 
\begin{proof} 
The $E_2$-term is given by 
$E(\alpha') \otimes P(\beta') \otimes P(t^{\pm 1})$, where 
$\alpha', \beta'$ are in 
\Cref{dRofHP}, and $|t| = 2$ has weight zero. 
Since $t$ comes from the cohomology of the circle, it is a permanent cycle. 
Since $\alpha', \beta'$ have weight $p$, one checks that $\alpha', \beta'$
are forced to be permanent cycles for weight reasons. Thus, there is no room for
differentials in the spectral sequence. 
\end{proof} 

\begin{remark} 
\label{twodifferentHP}
It follows that $\HP(B \alpha_p/k) \not\simeq \HH(B\alpha_p/k)^{tS^1}$. In
fact, the degenerate de Rham--$\mathrm{HP}$ spectral
sequence shows that $H^*(\mathrm{HP}(B\alpha_p/k))$ is uncountably dimensional in
each degree. However, since $\mathrm{HH}(B \alpha_p/k)$ is coconnective and
countably dimensional, it is easy to see that 
$\HH(B
\alpha_p/k)^{tS^1}$ is countably dimensional in each degree. 
\end{remark}

\begin{proposition} 
\label{cryscohbalphap}
    The crystalline cohomology of $B\alpha_p$ is given by
\[ H^*_\crys(B\alpha_p) \simeq W(k)[\beta']/p \beta',\]
where $|\beta'| = 2$.
\end{proposition} 
\begin{proof} 
This follows from Proposition~\ref{dRofHP}, provided that we can show that
$H^2_\crys(B\alpha_p)$ is simple $p$-torsion. But this is clear as the group
scheme $\alpha_p$ is annihilated by $p$.
\end{proof} 

Let us collect everything we know.

\begin{theorem} 
    For $B\alpha_p$, the following assertions hold true.
\begin{enumerate}
    \item[{\rm (a)}] The HKR spectral sequence does not degenerate.
        There is a nonzero differential $d_p\colon H^0( B \alpha_p, L_{B\alpha_p}) \to H^p( B
        \alpha_p, \bigwedge^p L_{B\alpha_p})$. 
    \item[{\rm (b)}] The Hodge--de Rham spectral sequence does not degenerate:
        there is a nonzero $d_1\colon H^1( B \alpha_p, \mathcal{O}) \to H^1( B
        \alpha_p,L_{B\alpha_p})$. Similarly, the conjugate spectral sequence
        does not degenerate: there is a nonzero differential $d_2\colon H^0(B\alpha_p, L_{B\alpha_p/k}) \to H^2(B\alpha_p, \mathcal{O})$.
    \item[{\rm (c)}] The Tate spectral sequence for $\HP$ does not degenerate (there is already
        a nonzero $d_2$).
    \item[{\rm (d)}] The de Rham--$\HP$ spectral sequence degenerates.
    \item[{\rm (e)}] The crystalline--$\TP$ spectral sequence degenerates.
\end{enumerate}
\end{theorem}
\begin{proof}
\begin{enumerate}
    \item[(a)] This was shown in the course of proving Proposition~\ref{HHofBalphap}.
    \item[(b)] This was shown in Proposition~\ref{dRofHP} and Remark~\ref{ConjBalphap}.
    \item[(c)] This was shown in Proposition~\ref{HHofBalphap}.
    \item[(d)] This was shown in Proposition~\ref{dRHpBalphap}.
    \item[(e)] All terms on the $E_2$-page live in even degrees
        (Proposition~\ref{cryscohbalphap}), so there are no differentials.
\end{enumerate}
\end{proof}

\section{A weak Lefschetz property}

Our main goal in this section is to verify a version of the weak Lefschetz
theorem for the Hodge cohomology of complete intersections in projective space (in arbitrary characteristic). In the
case of a smooth complete intersection, 
these results
are special cases of those
in~\cite[Expos\'e~XI.1.3]{sga72}. 
However, in the next section, it will be crucial to have the result for singular complete
intersections. 

For simplicity,
we work everywhere over a base field $k$. Unless specified
otherwise, cotangent complexes and de Rham cohomology are computed relative to
$k$.
\begin{definition}[Hodge $d$-equivalences]
    We say that a map of syntomic algebraic stacks $\Xscr\to\Yscr$ is a {\bf
	 Hodge $d$-equivalence} if, for each $s \geq 0$,
	 we have 
    $$\mathrm{cofib} \left(
	 R\Gamma(\Yscr,\bigwedge^s L_{\Yscr})\rightarrow R\Gamma(\Xscr,\bigwedge^s L_{\Xscr})
	 \right) \in 
	 D(k)^{\geq d-s}. $$
\end{definition}

\begin{remark}[Consequences in de Rham and crystalline cohomology]
    Say $\mathcal{X} \to \mathcal{Y}$ is a Hodge $d$-equivalence of syntomic
    stacks over a perfect field $k$ of characteristic $p$. Then the map
    $R\Gamma_{\dR}(\mathcal{X}) \to R\Gamma_{\dR}(\mathcal{Y})$ preserves the
    conjugate filtration, and each graded piece has cofiber in $D^{\geq d}(k)$
    by our assumption. It follows that the map $R\Gamma_{\dR}(\mathcal{X}) \to
    R\Gamma_{\dR}(\mathcal{Y})$ itself has cofiber in $D^{\geq d}(k)$. Passing
    to crystalline cohomology, this implies that  the cofiber $C$ of the map
    $R \Gamma_{\mathrm{crys}}(\Yscr) \rightarrow
    R\Gamma_{\mathrm{crys}}(\Xscr)$ is in $D(W(k))^{\geq d}$ and moreover
    $H^d(C)$ is $p$-torsion free.
\end{remark}

The main result of this section is the following result. 

\begin{proposition}
\label{Hodgedproj}
Let $X$ be a $d$-dimensional complete intersection in $\mathbb{P}^n_k$. 
Then the inclusion map $X \to \mathbb{P}^n_k$ is a Hodge $d$-equivalence. 
\end{proposition} 

The argument for 
\Cref{Hodgedproj} is based on an induction on the codimension and, in fact, it will be convenient
to prove a slightly stronger result (\Cref{maincorKAN}), based on the following two notions (which we
will need only for schemes). 

\begin{definition}[KAN $d$-equivalences]
Let $f\colon Y \to X$ be a map of syntomic $k$-schemes and let $\mathcal{L}$ be a line bundle
on $X$.  We say that $f$ is a {\bf KAN $d$-equivalence} if for each $s\geq 0$ and
    $r\geq 0$, 
    $$\mathrm{cofib} \left(
	 R\Gamma(X,\bigwedge^sL_X \otimes \mathcal{L}^{-r})\rightarrow R\Gamma(Y,\bigwedge^sL_Y
	 \otimes \mathcal{L}^{-r})
	 \right) \in D(k)^{\geq d-s}.$$
	 Taking $r=0$, we see that a KAN $d$-equivalence
is in particular a Hodge $d$-equivalence.
\end{definition}
\begin{definition}
    A {\bf Kodaira pair} is an $n$-dimensional $k$-scheme $Y$ and an ample line
	 bundle $\mathcal{L}$ such that
    $R\Gamma(Y,\bigwedge^sL_Y \otimes \mathcal{L}^{-r})\in D(k)^{\geq n-s}$ for all $s\geq 0$ and
    all $r>0$.\end{definition}

\begin{example}
\begin{enumerate}
    \item[(1)]
Projective space $\PP^n$ with any ample line bundle $\Oscr_{\PP^n}(h)$ (so that $h>0$)
    is a Kodaira pair.
\item[(2)]
        In characteristic zero, any smooth projective variety $Y$ with an ample line bundle $L$
    is a Kodaira pair, by the Kodaira--Nakano--Akizuki vanishing theorem (\cite[Corollary 2.11]{deligne-illusie}).
\end{enumerate}
\end{example}

\begin{proposition}[Weak Lefschetz] 
\label{weakL}
Let $(X, \mathcal{L})$ be a Kodaira pair of dimension $d$. If $i\colon H \to X$
is the inclusion of an effective 
Cartier divisor defined by a section of a positive power of $\mathcal{L}$, then
\begin{enumerate}
    \item[{\rm (i)}] the pair $(H, i^*\mathcal{L})$ is a Kodaira pair and
    \item[{\rm (ii)}] the inclusion $i\colon H \to X$ is a KAN $(d-1)$-equivalence.  
\end{enumerate}
\end{proposition} 
\begin{proof} 
In the following, we write $\mathcal{O}(r) = \mathcal{L}^{ r}$  for simplicity. 
For each $i \geq 0$, we consider the statements $S_i$:  
\begin{enumerate}[(a)]
    \item  $R \Gamma(H, \bigwedge^i L_H(-r)) \in D^{\geq d-i-1}(k)$ for $r > 0$ and
    \item the map $R \Gamma(X, \bigwedge^i L_X(-r)) \to R \Gamma(H, \bigwedge^i L_H(-r))$ has
        cofiber in 
        $D^{\geq d-i-1}(k)$ for $r \geq 0$. 
\end{enumerate}
For all $i \geq 0$, the statements $S_i$ imply the result. 
We will prove $S_i$ by induction on $i$. 
Note that in the statement $S_i$, part (a) is actually a consequence of part (b) since
$(X, \mathcal{L})$ is a Kodaira pair; however,
it will be convenient to have part (a) marked separately. 

For $i = 0$, 
we use the cofiber sequence $\mathcal{O}_X(-H-r) \to \mathcal{O}_X(-r) \to
i_*\mathcal{O}_H(-r)$ for any $r \geq 0$; we get 
\[ \mathrm{cofib}( R \Gamma(X, \mathcal{O}_X(-r)) \to 
R \Gamma(H, \mathcal{O}_H(-r))) \we  R \Gamma( X, \mathcal{O}_X(-r-H))[1] \in 
D^{\geq d-1}(k)
\]
by our assumption that $(X, \mathcal{L})$ is a Kodaira pair and that $\mathcal{O}(-H) =
\mathcal{L}^{-t}$ for some $t > 0$. 
This implies $S_0$. 

For $i > 0$, we 
first consider the factorization of the map in question (for any $r \geq 0$), 
$$R \Gamma(X, \bigwedge^i L_X(-r)) \xrightarrow{f_{i, r}} 
R \Gamma(H, i^* \bigwedge^i L_X(-r)) \xrightarrow{g_{i,r}} R \Gamma(H,
\bigwedge^i L_H(-r)).$$
It suffices to see that each of these maps has cofiber in $D^{\geq d-i-1}(k)$. 
The first map $f_{i,r}$ has cofiber in 
$D^{\geq d-i-1}(k)$
via the cofiber sequence
\[  R \Gamma(X, \bigwedge^i L_X(-r-H)) \to R \Gamma(X, \bigwedge^i L_X(-r))
\xrightarrow{f_{i,r}} 
R \Gamma(H, i^* \bigwedge^i L_X(-r))\]
and our assumption that $(X, \mathcal{L})$ is a Kodaira pair. 
For the second map $g_{i,r}$, we  
use the conormal sequence 
$\mathcal{O}_H(-H) \to i^* L_Y \to L_H$
in $D(H)$ to regard $i^* L_Y$ as a (two-term) filtered object in $D(H)$.
We can take exterior powers to obtain 
a filtration on $\bigwedge^i i^* L_Y$ (cf. the proof of \cite[Prop.
25.2.4.1]{SAG}); because $\mathcal{O}_H(-H)$ has rank one,
this filtration degenerates to 
a cofiber sequence 
$ \bigwedge^{i-1} L_H(-H) \to i^* \bigwedge^i L_X \to \bigwedge^i L_H$.  
Twisting by $-r$ and taking global sections, we obtain 
a cofiber sequence
\[  
R \Gamma( H, i^* \bigwedge^i L_X(-r)) \xrightarrow{g_{i,r}}R\Gamma( H, \bigwedge^i L_H(-r))
\to R \Gamma(H,  \bigwedge^{i-1} L_H(-H-r))[1]. 
\]
Now part (a) of statement $S_{i-1}$
implies that the cofiber of $g_{i,r}$ belongs to $D^{\geq d-i-1}(k)$, as desired. 
This completes the proof of the statement $S_i$ and thus of the result. 
\end{proof}

\begin{corollary} 
Let $(X, \mathcal{L})$ be a Kodaira pair. 
Let $i\colon Y \hookrightarrow X$ be a $d$-dimensional complete intersection of sections of
powers of $\mathcal{L}$ (in
particular, those sections form a regular sequence). 
Then $(Y, i^* \mathcal{L})$ is a Kodaira pair and $i\colon Y \to X$ is a KAN $d$-equivalence
(in particular, a Hodge $d$-equivalence). 
\label{maincorKAN}
\end{corollary} 
\begin{proof} 
Observe that the composite of a KAN-$m$-equivalence and a
KAN-$n$-equivalence is a KAN-$\mathrm{min}(m, n)$-equivalence. 
Therefore, the result follows by iteratively applying the weak Lefschetz (\Cref{weakL}). \end{proof}

\begin{example}[Smooth complete intersections]
    Let $X$ be a smooth complete intersection of dimension $d$ inside $\PP^n$
    and let $i\colon H\hookrightarrow X$ be a smooth hypersurface. 
	According to \Cref{maincorKAN}, $X, H$ are Kodaira pairs with respect to the
	line bundle $\mathcal{O}(1)$, 
and $i$ is a KAN $(d-1)$-equivalence. In fact, this is well-known: in characteristic zero, this follows from the Kodaira--Akizuki--Nakano vanishing theorem  (\cite[Corollary 2.11]{deligne-illusie}). In positive characteristic, one can use Deligne's computations to obtain the same result; see~\cite[Expos\'e~XI.1.3]{sga72}.
\end{example}
In the remainder of the section, we record two more basic properties of Hodge
cohomology and Hodge $d$-equivalences; these will be used essentially in the
next section.

\begin{proposition}[Preservation of Hodge $d$-equivalences]\label{prop:dproducts} 
\label{prop:quotient}
    Suppose that $X \to Y$ is a Hodge $d$-equivalence of syntomic $k$-schemes.
\begin{enumerate}
    \item[{\rm (1)}]  
For any  syntomic $k$-scheme $Z$, 
    $X \times_k Z \to Y \times_k Z$ is a Hodge $d$-equivalence. 
\item[{\rm (2)}]
If $G$ is  an affine  $k$-group scheme of finite type (and thus $G$ is syntomic) that acts on both $X$ and $Y$ equivariantly for the map $X \to Y$, then 
the map 
$[X/G] \to [Y/G]$
of quotient stacks is a Hodge $d$-equivalence. 
\end{enumerate}
    \end{proposition} 

\begin{proof} 
Part (1) follows from the K\"unneth formula in Hodge cohomology; 
we have
    \[R\Gamma(X \times_k Z, \bigwedge^sL_{X\times_k Y}) 
    \simeq \bigoplus_{a + b = s}R\Gamma(X,\bigwedge^aL_X) \otimes_k R\Gamma( Z,
            \bigwedge^bL_Z), 
    \]
    and similarly for $Y \times_k Z$.
    Note now that $R\Gamma(Z,\bigwedge^bL_Z)$ belongs to $D(k)^{\geq -b}$
    because $Z$ is syntomic. Therefore, the map 
    \(R\Gamma(X, \bigwedge^aL_X) \otimes_k   R \Gamma( Z,
            \bigwedge^bL_Z)
    \to 
    R \Gamma(Y, \bigwedge^aL_Y) \otimes_k
    R \Gamma( Z, \bigwedge^bL_Z)
    \)
    has cofiber in $D(k)^{\geq d -  a- b}  = D(k)^{\geq d-s}$, as desired. 

For part (2), consider the resolution $\cdots\triplearrows G\times X\rightrightarrows X$ of the stack
    $[X/G]$ and similarly for $[Y/G]$. By hypothesis and Proposition~\ref{prop:dproducts}, 
    the induced map of cosimplicial objects $R\Gamma(G^\bullet\times_k
    Y,\bigwedge^sL_{G^\bullet\times_kY})\rightarrow R\Gamma(G^\bullet\times
    X,\bigwedge^sL_{G^\bullet\times_k X})$ has levelwise
    cofiber in $D(k)^{\geq d-s}$. Since $D(k)^{\geq d-s}$ is closed under limits
	 in $D(k)$, the result now follows by taking the limit. \end{proof} 

\begin{proposition}[Projective bundle formula] 
Let $\Xscr$ be a syntomic stack. 
Given an $n$-dimensional vector bundle $V$ over $\Xscr$, let $\Yscr$ be the
associated projective bundle over $\Xscr$. Then 
there exists a class $c_1 \in  H^1( \Yscr, L_{\Yscr})$  
such that 
$H^*( \Yscr, \bigwedge^{\ast} L_{\Yscr})$ is a free module over 
$H^*( \Xscr, \bigwedge^{\ast} L_{\Xscr})$ on $1, c_1, \dots,
c_1^{n-1}$. 
\label{prop:pbf}
\end{proposition} 
\begin{proof} 
The class $c_1$ is the first Chern class (in Hodge cohomology) of the tautological line bundle
$\mathcal{O}(1)$ on $\Yscr$, which is defined via pullback from the
induced map $\Yscr \to B \mathbb{G}_m$ classifying $\mathcal{O}(1)$. 
The result then asserts that for each $i$, 
the map 
\[ \bigoplus_{j = 0}^{n-1} R \Gamma( \Xscr, \bigwedge^{i-j}
L_{\Xscr})[-j] \to R \Gamma(\Yscr, \bigwedge^i L_{\Yscr})
,\]
obtained  as multiplication by $c_1^j$ on the $j$th factor, is an equivalence. 
This assertion is local on $\Xscr$, whence we reduce to the case of
$\Xscr$ an affine scheme and $V$ a trivial bundle, for which the result is classical. 
\end{proof}

\section{Approximation of classifying spaces and failure of
HKR}\label{sec:approximation}

Using our study of Hodge $d$-equivalences from the previous section, we prove
Theorem~\ref{thm:approximationintro} from the introduction. The ideas here are
not new and go back to work of Serre~\cite{serre} and Totaro~\cite{totaro}; we
also use arguments from \cite{bms1}.

\begin{proof}[Proof of Theorem~\ref{thm:approximationintro}]
First, we assume that $G$ is a finite group scheme.
    We claim that   there is a finite dimensional representation
    $V$ of $G$ and a $d$-dimensional complete intersection $Z\subseteq\PP(V)$
    such that $Z$ is stable under the $G$-action, $G$ acts freely on $Z$, and
    $Z/G \simeq [Z/G]$ is smooth and projective. This is a standard argument involving an application of a Bertini-type theorem for the quotient variety $\PP(V)/G$, see, e.g., \cite[2.7-2.9]{bms1}.\footnote{The argument relies on Bertini-type
	 theorems; in case $k$ is finite, one can use
	 Bertini theorems in the form of \cite{poonen}.} 
    Having found $Z$, we note that $Z\hookrightarrow\PP(V)$ is a Hodge
    $d$-equivalence by Proposition~\ref{Hodgedproj}. Therefore, the induced map
    $Z/G \simeq [Z/G] \hookrightarrow [\PP(V)/G]$ on quotient stacks is a Hodge $d$-equivalence by
	 \Cref{prop:quotient}. The theorem now
    follows from Proposition~\ref{prop:pbf} by taking $X=Z/G$.

Now suppose $G$ is geometrically reductive. 
For each $r$, we let $G_r \subset G$ be the kernel of the $r$th Frobenius on
$G$. 
According to \cite[Cor. II.4.12]{Jantzen}, 
for any finite-dimensional $G$-representation 
$V$, we have that 
the cohomology groups $H^i(G, V), H^i(G_r, V)$ are finite-dimensional for $i \geq 0$, and 
$H^i(G, V) \simeq \varprojlim_r H^i(G_r, V)$. 
We claim that for all $i, j, r \geq 0$,  the vector spaces
$H^i(BG, \wedge^j L_{BG})$ and $H^i(  BG_r, \wedge^j, L_{BG_r})$ are
finite-dimensional, and 
\begin{equation} \label{limofFrobker} H^i(BG, \wedge^j L_{BG}) \simeq \varprojlim 
H^i(  BG_r, \wedge^j L_{BG_r}).\end{equation} By finite-dimensionality, this implies
that for any $i, j$, 
the map $H^i( BG, \wedge^j L_{BG}) \to H^i(BG_r, \wedge^j L_{BG_r})$ is
injective
for $r \gg 0$. From this, we reduce the case of reductive $G$ to finite $G$
treated above. 

To prove the claim \eqref{limofFrobker}, we observe that by functoriality we have maps 
(where each object belongs to 
the appropriate derived category, and the maps are compatible in the natural
sense)
$$\mathrm{coLie}(G) \to \dots \to \mathrm{coLie}(G_{r+1}) \to \mathrm{coLie}(
G_r) \to \dots \to \mathrm{coLie}(G_1). $$
By the description of the $G_r$ as Frobenius kernels,  we find that 
each of these maps 
is an isomorphism on $H^0$, and induces the zero map on $H^{-1}$. 
Taking rational cohomology, we find easily that 
\begin{equation} \label{invlimproj} \varprojlim R \Gamma( G_r, \mathrm{Sym}^i \mathrm{coLie}(G_r)) \simeq 
\varprojlim R \Gamma( G_r, H^0(\mathrm{Sym}^i \mathrm{coLie}(G_r))) 
= 
\varprojlim R \Gamma( G_r, \mathrm{Sym}^i \mathrm{coLie}(G))
, \end{equation}
and all cohomologies are finite-dimensional in each degree. Combining with 
\cite[Cor. II.4.12]{Jantzen} and the decomposition \eqref{BGhodgecoh}, we
conclude 
$R \Gamma( BG, \wedge^i L_{BG}) \simeq \varprojlim_r R \Gamma(BG, \wedge^i
L_{BG_r})$.
This yields 
\eqref{limofFrobker}, since 
finite-dimensionality prevents the existence of nonzero $\lim^1$ terms. 
\end{proof}

\begin{remark}
The proof of Theorem~\ref{thm:approximationintro} given above proves a slightly stronger statement in the case of finite group schemes $G$: we find approximations $X \to BG$ as in Theorem~\ref{thm:approximationintro} such that $H^*(X,\bigwedge^{*'}L_X)$ is free as a bigraded $H^*(BG,\bigwedge^{*'}L_{\B G})$-module in total degrees $*+*'\leq d$ on classes $c^i\in H^i(X,\bigwedge^iL_X)$ for $0\leq i\leq\lfloor\tfrac{d}{2}\rfloor$.
\end{remark}

We can prove now that several spectral sequences as explained earlier are
non-degenerate. 

\begin{proof}[Proof of Theorem~\ref{thm:intro1}]
    Choose $G=\alpha_p$ or $G=\mu_p$. In each case, we found that there is a
    nonzero differential $d_{p}\colon H^0(B G,\bigwedge^1 L_{BG})\rightarrow H^{p}(BG,\bigwedge^p L_{BG})$ in the HKR spectral
    sequence. If we choose $X\rightarrow BG$ as in
    Theorem~\ref{thm:approximationintro} to have
    dimension $2p$, then we find that the differential
    $d_{p}\colon H^0(X,\Omega^1_X)\rightarrow H^p(X,\Omega^p_X)$ is nonzero, as desired.
\end{proof}

\begin{theorem}\label{thm:other}
    Let $k$ be a perfect field of characteristic $p>0$.
    \begin{enumerate}
        \item[{\rm (a)}] There exists a smooth projective $2p$-fold such that the
            de Rham--$\HP$ spectral sequence does not degenerate.
        \item[{\rm (b)}] There exists a smooth projective variety such that
            the crystalline--$\TP$ spectral sequence does not degenerate.
        \item[{\rm (c)}] There exists a smooth projective variety $X$ such that
            the filtration on $H^*(\TP(X))$ arising from the crystalline--$\TP$
            spectral sequence is not split.
    \end{enumerate}
\end{theorem}

\begin{proof}
    For (a), let $X\rightarrow B\mu_p$ be a smooth projective approximation
    as in Theorem~\ref{thm:approximationintro} for $d=2p$, so
    $H^t(B\mu_p,\bigwedge^s L_{B\mu_p})\rightarrow H^t(X,\bigwedge^s L_X)$ is
    injective for $s+t\leq 2p$. By the results of Section~\ref{sec:bmup}, we
    see that in the de Rham--$\HP$ spectral sequence for $B\mu_p$, there is a
    nonzero differential    $d_p\colon H^1_{\mathrm{dR}}({B\mu_p})\rightarrow H^{2p}_{\mathrm{dR}}({B\mu_p})$.
    Thus, by naturality, there is also a nonzero differential in the de
    Rham--$\HP$ spectral sequence for $X$, as desired.

    Now, for part (b), by Lemma~\ref{lem:mupmup}, the crystalline--$\TP$ spectral
    sequence for $B(\mu_p\times\mu_p)$ does not degenerate. In particular,
    from the proof we see that
    there is some first nonzero differential
	 $d_r\colon H^1_{\mathrm{crys}}(B(\mu_p\times\mu_p))\rightarrow H^{2r}_{\mathrm{crys}}(B
	 ( \mu_p \times \mu_p))$.
    Choose a smooth projective approximation $X\rightarrow B(\mu_p\times\mu_p)$
    with $d=2r$. Then, the crystalline--$\TP$ spectral sequence does not
    degenerate for $X$.

    To prove part (c), we take a smooth projective approximation to
    $B\mu_p$ as in~\ref{thm:approximationintro} with $d > 2p$. Then,
    $H^t_{\mathrm{crys}}({B\mu_p})\rightarrow H^t_{\mathrm{crys}}(X)$ is injective for
    $t\leq 2p$. 
	In particular, in this range, the image of the classes in  
    $H^t_{\mathrm{crys}}({B\mu_p})$ in $H^t_{\mathrm{crys}}(X)$ are permanent
    cycles. Since the extensions in the spectral sequence are nontrivial for $B
    \mu_p$, they also must be nontrivial for $X$. 
\end{proof}

\begin{remark}[Obtaining liftable examples]
Take $G=\mu_p$ in the proof of Theorem~\ref{thm:intro1} given above. Since the
    group scheme $\mu_p$ admits a lift (even a unique one) to $W_2(k)$, one can
    show that the smooth projective variety $X$ used in the proof of Theorem~\ref{thm:intro1} (coming from Theorem~\ref{thm:approximationintro}) also lifts to $W_2(k)$, thus yielding a liftable example where the HKR spectral sequence does not degenerate. A similar remark applies to Theorem~\ref{thm:other}.
\end{remark}

\section{Hochschild homology and the $\dlog$ map}

The construction of Hochschild homology and the HKR spectral sequence allows for
a twisted version, via an Azumaya algebra or a class in the Brauer group. In this section, we 
describe this spectral sequence.

Throughout, we fix a base commutative ring $k$. 
The starting point is the following result from \cite{cortinas-weibel}, stating
that Hochschild homology groups cannot distinguish between Azumaya algebras over affine schemes. 

\begin{theorem}[Corti\~{n}as--Weibel~\cite{cortinas-weibel}] 
\label{CWthm}
Suppose $R$ is a $k$-algebra and $\mathcal{A}$ is an Azumaya $R$-algebra. 
Then there is a functorial (in $R, \mathcal{A}$) isomorphism  
of $H^*(\HH(R/k))$-modules, $H^*(\HH(\mathcal{A}/k))\simeq H^*(\HH(R/k))$. 
\end{theorem} 

\begin{corollary}[The twisted HKR theorem] 
\label{twistedHKRaffine}
Let $R$ be a smooth $k$-algebra and let $\mathcal{A}$ be an Azumaya $R$-algebra. 
Then there is a natural isomorphism 
$H^*(\HH(\mathcal{A}/k))\simeq \Omega^{-\ast}_{R/k}$.  
\end{corollary}

The isomorphism $H^*(\HH(\mathcal{A}/k))\simeq H^*(\HH(R/k))$ appearing in Corollary~\ref{twistedHKRaffine} can be chosen functorially at the level of cohomology groups, but not at the level of complexes. Globalizing, this leads to the following constructions. 

\begin{construction}[Twisted Hochschild homology] 
Let $X$ be a $k$-scheme and let $\mathcal{A}$ be an Azumaya
algebra over $X$ (i.e., a sheaf of Azumaya $\mathcal{O}_X$-algebras).
For each  \'etale map  from an affine, $\Spec(R)\to X$, we obtain an Azumaya
$R$-algebra $\mathcal{A}_R$, and can form the Hochschild homology $\HH(
\mathcal{A}_R/k)$. As $R$ varies, we  obtain an object of $D(X)$,
denoted $\ShHH$.
We write $\HH(\mathcal{A}/k)=R\Gamma(X,\ShHH)$ for the global sections of 
$\ShHH$. 
We call this construction the \textbf{$\mathcal{A}$-twisted
Hochschild homology of $X$}. 
\end{construction} 

\begin{construction}[The twisted HKR spectral sequence] 
Let $\mathcal{A}$ be an Azumaya algebra over the smooth $k$-scheme $X$. 
For each \'etale map 
$\Spec(R) \to X$, we  have functorial isomorphisms of Hochschild homology groups 
$H^*(\HH(\mathcal{A}_R/k)) \simeq \Omega^{\ast}_{R/k}$  by
\Cref{twistedHKRaffine}.  
Globalizing, we obtain a spectral sequence 
\begin{equation} \label{twistedHKRsseq} E_2^{s,t}=H^s( X, \Omega^{-t}_{X/k})
    \Rightarrow
H^{s+t}(\HH(\mathcal{A}/k)),  \end{equation}
the {\bf $\Ascr$-twisted HKR spectral sequence}.
\end{construction} 

The main result of this section (\Cref{prop:differential}) is an identification of the first differential
in this spectral sequence. In proving the result, we will also 
clarify the 
precise choice of isomorphism 
in \Cref{CWthm}. 

\begin{construction}[\'Etale twisted $K$-theory] 
Let $\mathcal{A}$ be an Azumaya algebra over the $k$-scheme $X$, representing a
Brauer class $\alpha$. 
We consider 
the \textbf{$\mathcal{A}$-twisted \'etale $K$-theory}
$\ShK$, defined as the \'etale sheafification over $X$ of $R \mapsto\K(
\mathcal{A}_R)$, where $\K(\mathcal{A}_R)$ denotes the  algebraic
$\K$-theory spectrum of $\mathcal{A}_R$. Since Hochschild homology has \'etale descent \cite{WG,
cortinas-weibel}, the object $\ShK$  is equipped with a trace map
(obtained by sheafifying the Dennis trace)
$$\ShK \to
\ShHH.$$  
\end{construction} 

\begin{construction}[The normalization of the isomorphism $\HH_*(R/k) \simeq
\HH_*(\mathcal{A} /k)$]
Let $\mathcal{A}$ be an Azumaya algebra over the $k$-scheme $X$. 
Here we construct explicitly the isomorphism of \Cref{CWthm}. 

According to the \'etale descent theorem \cite{WG, cortinas-weibel}, 
it follows that 
$H^*( \ShHH)$ define quasicoherent sheaves on the \'etale site of $X$, which
consequently have no higher cohomology. 
According to \cite[Sec. 5]{{antieau-cohomological}}, the \'etale sheafified homotopy
groups $\pi_i \ShK$ are canonically isomorphic to 
the untwisted sheafified homotopy groups
$\pi_i \mathbf{K}( -)$. 
In particular, we have a canonical isomorphism 
$\pi_0 \ShK \iso\mathbb{Z}$. The trace map thus gives a map $\mathbb{Z} \to \pi_0 ( \ShHH))$
which one checks \'etale locally (over which $\mathcal{A}$ is trivial) to  be a generator, i.e., to induce an isomorphism 
$H^*(( \mathbf{HH}(-/k)) \iso H^*( \ShHH)$. 
\end{construction} 

Now we identify the first differential in the $\Ascr$-twisted HKR spectral sequence. 
For this, we use the 
map $\mathrm{dlog}\colon \mathbb{G}_m \to \Omega^1_{(-)/k}$ of \'etale sheaves. 
In particular, for any scheme $X$, it defines
a map $H^i( X, \mathbb{G}_m) \to H^i (X,\Omega^1_{X/k})$. 
This map arises explicitly in the Dennis trace map. 
Namely, if $R$ is a smooth $k$-algebra, then the map 
\[ R^{\times} \to K_1(R) \to  H^{-1}(\HH(R/k))  \simeq \Omega^1_{R/k}\]
is given by $\mathrm{dlog}$, in view of \cite[Theorem 6.2.16]{Rosenberg}. 

\begin{proposition}[The first differential in twisted HKR]\label{prop:differential}
    Let $\Ascr$ be an Azumaya algebra on a $k$-scheme $X$ with class $\alpha \in
	 H^2(X, \mathbb{G}_m)$. Then,
    the differential $$d_2^\alpha\colon H^0(X,\Oscr_X)\rightarrow H^2(X,\Omega^1_{X/k})$$ in the
    $\Ascr$-twisted HKR spectral sequence for $\HH(\Ascr/k)$ sends $1$ to $\dlog\alpha$.
\end{proposition}

\begin{proof}
    We have a natural trace map of \'etale sheaves $\ShK\rightarrow\ShHH$
    and hence an induced map of \'etale descent spectral sequences. The
    argument of~\cite[Proposition~5.1]{antieau-cohomological} provides
	 isomorphisms $\pi_i\ShK\iso\pi_i\mathbf{K}^\et(-)$ and \Cref{twistedHKRaffine}
	 provides isomorphisms
     $H^i(\ShHH)\iso H^i(\mathbf{HH}(-/k))$; by construction, these are
	 compatible with the $\alpha$-twisted and untwisted trace. Thus, $\pi_0\ShK\iso\ZZ$,
     $\pi_1\ShK\iso\Gm$, $H^0(\ShHH)\iso\Oscr_X$ and $H^{-1}(\ShHH)\iso\Omega^1_X$.
    Since the trace induces $\mathrm{dlog}$ in degree $1$, we have a commutative square $$\xymatrix{
    H^0(X,\ZZ)\ar[r]^{d_2^\alpha}\ar[d]&H^2(X,\Gm)\ar[d]^{\dlog}\\
    H^0(X,\Oscr_X)\ar[r]^{d_2^\alpha}&H^2(X,\Omega^1_X),
    }$$
    where $d_2^\alpha$ denotes the differential in the spectral sequences
    converging to the $\alpha$-twisted forms of \'etale $K$-theory and
    Hochschild homology. By the main result of~\cite{antieau-cech}, we have
    $d_2^\alpha(1)=\alpha$ for the top horizontal arrow. Since the left
    vertical arrow sends $1$ to $1$, the corollary follows.
\end{proof}

\begin{remark}
    Let $k$ be a perfect field of characteristic $p>0$.
    If $X$ is a smooth proper $k$-scheme and $\Ascr$ is an Azumaya algebra on
    $X$ with Brauer class $\alpha\in H^2(X,\Gm)$ such
    that $\dlog\alpha\neq 0$ in $H^2(X,\Omega^1_X)$, then the
    $\Ascr$-twisted HKR spectral sequence does not
    degenerate, contrary to what happens in the untwisted case when $\dim X\leq
    p$ (see~\cite{antieau-vezzosi}). In the next two sections, we will  find examples.
\end{remark}

\begin{remark}[Cohomological Brauer classes]
    The $\Ascr$-twisted Hochschild homology and HKR spectral sequence only depend on the Brauer class
    $\alpha\in H^2(X,\Gm)$, and one can define twisted Hochschild homology
	 purely in terms of $\alpha$ (even when it is not representable by an
	 Azumaya algebra). 
Indeed, one can construct twisted Hochschild homology as the Hochschild
homology of representing derived Azumaya algebras
\cite{Toen12, AG14}. 
	 Moreover, twisted Hochschild homology and
    the associated twisted HKR spectral sequence (along the lines of
    Definition~\ref{def:sss}(a)) exist for any class
    $\alpha\in H^2(\Xscr,\Gm)$ for any syntomic $k$-stack $\Xscr$.
\end{remark}

\section{$\PGL_n$}
\label{sec:PGLn}

We will give two different approaches to constructing counterexamples to the
degeneration of the twisted HKR spectral sequence as in
Theorem~\ref{thm:intro2}. The first, and most naive, is to take suitable
approximations to $B\PGL_p$. This method produces smooth projective $3$-folds.
Second, for $p=2$, we note that classical Enriques surfaces give examples. 
In this section, we describe the first approach. 

For the following argument, it will be convenient to use (very mildly) the language
of higher stacks: in particular, we will want to consider $B^2 \mathbb{G}_m$ as
a higher stack, and regard Hodge cohomology as sheaf cohomology. 
We briefly review this language below (in the form of sheaves of spaces). 
For simplicity, we will restrict to smooth schemes, since this is all we will
need henceforth. 

\begin{definition}[Sheaves of spaces or higher stacks] 
Let $k$ be a perfect field. 
Consider the category $\mathrm{Sm}_k^{\mathrm{aff}}$ of affine smooth $k$-schemes, equipped with the
\'etale topology. 
We can consider the $\infty$-category $\mathrm{Shv}( \mathrm{Sm}_k^{\mathrm{aff}})$ of sheaves of
spaces on $\mathrm{Sm}_k^{\mathrm{aff}}$; see \cite{HTT} for a detailed treatment of sheaves
of spaces. 
Since $\mathrm{Shv}( \mathrm{Sm}_k^{\mathrm{aff}})$ is an $\infty$-category, there is a
well-defined homotopy type $\mathrm{Map}_{\mathrm{Shv}( \mathrm{Sm}_k^{\mathrm{aff}})}(X, Y)$
for any two objects $X, Y \in \mathrm{Shv}( \mathrm{Sm}_k^{\mathrm{aff}})$. 
\end{definition} 

\begin{example}[Examples of higher stacks] 
    \begin{enumerate}[(a)]
    \item Any smooth algebraic stack $\Xscr$ over $k$ yields (and is determined by)
an object 
$\mathrm{Shv}( \mathrm{Sm}_k^{\mathrm{aff}})$ via the groupoid-valued functor of points. 
\item  
For any smooth commutative group scheme $G$ over $k$ and any $n \geq 0$, we obtain a higher stack
$K(G, n) \in \mathrm{Shv}( \mathrm{Sm}_k^{\mathrm{aff}})$. Explicitly, $K(G, n)$ is the \'etale sheafification of the functor 
which sends a smooth $k$-algebra $R$ to the Eilenberg--MacLane space $K(G(R), n)$. 
\item 
Fix $i \geq 0$. Consider the functor $\Omega^i$ which sends a smooth $k$-algebra $R$ to the
differential forms $\Omega^i_R$. 
For $n \geq 0$, we obtain a sheaf of spaces on $\mathrm{Sm}_k^{\mathrm{aff}}$, 
$K( \Omega^i, n) \in \mathrm{Shv}( \mathrm{Sm}_k^{\mathrm{aff}})$,
such that 
$K( \Omega^i, n)(R)$ is the Eilenberg--MacLane space 
$K( \Omega^i_R, n)$; here we do not need to \'etale sheafify further. 
\end{enumerate}
\end{example} 

The language of higher stacks will be 
useful for us because one has a good internal theory of cohomology. 

\begin{construction}[Higher group cohomology] 
Let $G$ be a smooth commutative $k$-group scheme. 
Given an object $T \in 
\mathrm{Shv}( \mathrm{Sm}_k^{\mathrm{aff}})$, we have the 
\textbf{cohomology groups} (for $n \geq 0$)
$$H^n( T, G) = \pi_0 \mathrm{Map}_{\mathrm{Shv}( \mathrm{Sm}_k^{\mathrm{aff}})}( T, K(G, n));$$
this agrees with the usual definition (\'etale cohomology) in case $T$ is
representable by a smooth $k$-scheme (cf. also \cite[Cor. 2.1.2.3]{SAG}).
By delooping, we can regard these as the cohomology groups of a complex $R
\Gamma(T, G)$. 
\end{construction} 

\begin{construction}[Hodge cohomology as sheaf cohomology] 
Let $\Xscr$ be a smooth algebraic stack over $k$, which we regard as an
object of $\mathrm{Shv}( \mathrm{Sm}_k^{\mathrm{aff}})$ as above. 
Then we have
natural isomorphisms
\[ H^n( \Xscr, \bigwedge^i L_{\Xscr}) = 
\pi_0 \mathrm{Map}_{\mathrm{Shv}( \mathrm{Sm}_k^{\mathrm{aff}})}( \Xscr, K( \Omega^i,
n)).
\]
\end{construction} 

\begin{example}[The degree two class on $B\PGL_n$] 
We have a natural map $B\PGL_n \to K(\mathbb{G}_m, 2)$. 
Explicitly, $B\mathbb{G}_m$ defines a commutative group object of 
$\mathrm{Shv}( \mathrm{Sm}_k^{\mathrm{aff}})$ acting centrally on $B\GL_n$;
$B\PGL_n$ is the quotient of $B\mathbb{G}_m$ acting on 
$B\GL_n$. We thus obtain a fiber sequence
$B\GL_n \to B\PGL_n \to K(\mathbb{G}_m, 2)$, i.e., a class in $H^2(B\PGL_n,
\mathbb{G}_m)$. 
\end{example} 

\begin{example}[The $\mathrm{dlog}$ map and its delooping] 
For any smooth $k$-algebra $R$, we have a homomorphism of abelian groups $\mathrm{dlog}\colon \mathbb{G}_m(R) \to \Omega_R^1$ given by $\mathrm{dlog}(f) = \frac{df}{f}$. Sheafifying and delooping twice yields a natural map 
$B^2 \mathrm{dlog}\colon K(\mathbb{G}_m, 2) \to K( \Omega^1, 2)$. 
\end{example} 

\begin{proposition} 
\label{notnullhomotopicdlog}
If $n$ is divisible by $p$,
then the composite map $B\PGL_n \to K( \mathbb{G}_m, 2)
\xrightarrow{B^2 \mathrm{dlog}} K(\Omega^1, 2)$ is
not nullhomotopic in $\mathrm{Shv}( \mathrm{Sm}_k^{\mathrm{aff}})$. 
\end{proposition} 
\begin{proof} 
The map $B^2 \mathrm{dlog}$ is clearly not nullhomotopic (since its double looping is
nonzero). Our claim is then that the map 
$H^2( K(\mathbb{G}_m, 2), \Omega^1) \to 
H^2( B\PGL_n, \Omega^1) = H^2( B\PGL_n, L_{B\PGL_n})$ is
injective. 
Indeed, the map $B\PGL_n \to K( \mathbb{G}_m, 2)$ is obtained from the map 
$B\GL_n \to \ast $ by taking homotopy orbits by $B \mathbb{G}_m$. 
Thus, we have a commutative diagram
of fiber sequences,
\begin{equation} \label{fibseq}\begin{gathered} \xymatrix{
B\mathbb{G}_m \ar[d]  \ar[r]^= &  B \mathbb{G}_m \ar[d]  \\
B\GL_n \ar[r] \ar[d]   &  \ast \ar[d]  \\
B\PGL_n \ar[r] &  K( \mathbb{G}_m, 2).
}\end{gathered} \end{equation}

Given smooth stacks $\Xscr, \Yscr$ with 
$R \Gamma( \Xscr, \mathcal{O}) = R \Gamma(\Yscr, \mathcal{O}) =k$, the K\"unneth
formula in Hodge cohomology 
yields $R \Gamma( \Xscr \times \Yscr, L_{\Xscr \times \Yscr}) = 
R \Gamma( \Xscr, L_{\Xscr}) \oplus R \Gamma(\Yscr, L_{\Yscr})$. 
The diagram of fiber sequences expresses $B\PGL_n$ as the geometric
realization of a simplicial stack 
given as the bar construction of $B \mathbb{G}_m$ acting on $B\GL_n$, i.e., 
$\cdots\triplearrows  B\GL_n \times B\mathbb{G}_m
\rightrightarrows B\GL_n$; similarly, $K( \mathbb{G}_m, 2)$ is the bar
construction of $B \mathbb{G}_m$ acting on a point. 
Applying the cochains functor $R\Gamma(\cdot, \Omega^1)$, which carries
these simplicial resolutions to totalizations, we find that 
these two observations imply that both vertical sequences in \eqref{fibseq}
are carried to fiber sequences. 
That is, 
\begin{gather*} R\Gamma(B\PGL_n, \Omega^1) 
\we \mathrm{fib} ( R \Gamma( B\GL_n, \Omega^1) \to 
R \Gamma( B\mathbb{G}_m, \Omega^1)), \\
R\Gamma(K(\mathbb{G}_m, 2), \Omega^1) 
\we \mathrm{fib} ( 0 \to 
R \Gamma( B\mathbb{G}_m, \Omega^1))
.\end{gather*}

Now we have the following computation in Hodge cohomology: $R \Gamma( B\GL_n,
\Omega^1) = R \Gamma( B\GL_n, L_{B\GL_n}) \simeq
k[-1]$; this is a classical calculation, and a modern reference for a much more general statement is \cite[Theorem 10.2]{totaro}. Furthermore, the diagonal map $B \mathbb{G}_m \to B\GL_n$ induces the map $k[-1]
\to k[-1]$ given by multiplication by $n$. This follows from the explicit
expression of Hodge cohomology given in this case (i.e., as analogous to
singular cohomology), cf. \cite[Theorem
4.1]{totaro}. 
We obtain that $R\Gamma(B\PGL_n, \Omega^1)$ is the 
two-term complex $k \stackrel{n}{\to} k$ in degrees $1$ and $2$ 
and 
$R\Gamma(K(\mathbb{G}_m, 2), \Omega^1)  = k[-2]$. 
By naturality of the above diagrams, we obtain that the map 
$R\Gamma(K(\mathbb{G}_m, 2), \Omega^1) \to 
R\Gamma(B\PGL_n, \Omega^1)$ is an isomorphism in $H^2$:
the generating classes in both sides arise by coboundary from $H^1(B
\mathbb{G}_m, \Omega^1)$. 
\end{proof} 

Now, we can prove Theorem~\ref{thm:intro2} from the introduction.

\begin{proof}[Proof of Theorem~\ref{thm:intro2}]
    Take a smooth projective $X$ with $d=3$ in
    Theorem~\ref{thm:approximationintro}
    for $G=\PGL_p$. According to  
    \Cref{notnullhomotopicdlog}, 
    the composite map $B\PGL_p \to K( \mathbb{G}_m, 2) \to K( \Omega^1, 2)$ is
    not nullhomotopic and defines a nontrivial class (indeed, the generator) of
    $H^2(B\PGL_p, L_{\PGL_p})$. The pullback of this class under the map $X \to\PGL_p$
    is nonzero by construction. Therefore, by naturality, 
    the pullback of the canonical class in $H^2(X,\Gm)$ maps via
    $\dlog$ to an exact order $p$ class in $H^2(X,\Omega^1_X)$. We are done by
	Proposition~\ref{prop:differential}: there is a nonzero $d_2$-differential in the twisted HKR
	spectral sequence.
\end{proof}

\section{Classical Enriques surfaces in characteristic $2$}

Let $k$ be an algebraically closed field of characteristic $2$.
For background on Enriques surfaces in positive characteristic,
see~\cite{bombieri-mumford-3} and~\cite[Section~II.7.3]{illusie-derham-witt}. A smooth proper surface $X$ over $k$ is {\bf
Enriques}
if its canonical class, $\omega_X$, is algebraically equivalent to zero and if
its second Betti number is $B_2=10$. An Enriques surface is {\bf classical} if
$H^1(X,\Oscr_X)=0$. In this case, $H^2(X,\Oscr_X)=0$, $\omega_X$ is not
trivial, and $\omega_X^{\otimes 2}\iso\Oscr_X$.

Recall that for smooth schemes over any perfect field of characteristic $p>0$ we have exact sequences of \'etale sheaves $$1\rightarrow
\Gm\xrightarrow{p^n}\Gm\rightarrow\nu_n(1)\rightarrow 1$$ for all $n$.
These induce short exact sequences
$$0\rightarrow\Br(X)/p^n\rightarrow H^2(X,\nu_n(1))\rightarrow H^3(X,\Gm)[p^n]\rightarrow
0$$ for each $n$.

\begin{lemma}
    If $X$ is a classical Enriques surface over an algebraically closed field
    of characteristic $2$, then $\Br(X)=\ZZ/2$. In
    particular, the nonzero class $\alpha\in\Br(X)$ defines a nonzero class in
    $H^2(X,\nu_n(1))$ for all $n\geq 1$.
\end{lemma}

\begin{proof}
    The first part is precisely~\cite[Corollary~5.7.1]{cossec-dolgachev}.
\end{proof}

\begin{proposition}\label{prop:enriques}
    Let $X$ be a classical Enriques surface over an algebraically closed field
    $k$, and let $\alpha\in\Br(X)$ denote the nonzero class. Then, $\dlog\alpha\neq
    0$ in $H^2(X,\Omega_X^1)$.
\end{proposition}

\begin{proof}
    Because $k$ is algebraically closed and $X$ is a smooth projective surface, we have an exact sequence
    $$0\rightarrow \varprojlim H^2(X,\nu_\bullet(1)) \xrightarrow{d
	 \mathrm{log} ( [-])} H^2(X,W\Omega^1_X)\xrightarrow{1-F}H^2(X,W\Omega^1_X)\rightarrow
    0$$ by~\cite[5.22.5]{illusie-derham-witt}.

As $H^2(X,\mathcal{O}_X) = 0$ and $H^2(X, W\mathcal{O}_X) = 0$ (see Figures~\ref{fig:hodge} and~\ref{fig:hodge-witt} reproduced from~\cite[Section~II.7.3]{illusie-derham-witt}), we have a commutative diagram
    $$\xymatrix{
        H^3_{\crys}(X/W)\ar[r]\ar[d]&H^2(X,W\Omega^1_X)\ar[d]\\
        H^3_{\dR}(X/k)\ar[r]&H^2(X,\Omega^1_X),
    }$$ where the vertical maps are induced by killing $p$ in crystalline cohomology and the
    horizontal maps are the natural maps coming from the slope and Hodge filtrations. 
    All four maps in the diagram are isomorphisms for $X$ a classical Enriques surface (again, see Figures~\ref{fig:hodge} and \ref{fig:hodge-witt}). This shows that our class
    survives.
\end{proof}

Thus, the twisted HKR spectral sequence does not degenerate for the non-trivial
twist of a classical Enriques surface over an algebraically closed field of
characteristic $2$. More specifically, we can compute the Hochschild homology
exactly.

\begin{corollary}\label{cor:hh}
    Let $\Ascr$ be a geometrically non-trivial $2$-torsion Azumaya algebra on a classical Enriques
    surface $X$ in characteristic $2$. Then, $\HH(\Ascr/k)$ is discrete and
    $H^0(\HH(\Ascr/k))$ is a $12$-dimensional $k$-vector space.
\end{corollary}

\begin{proof}
    We can assume that $k$ is algebraically closed. 
	We obtain that $\HH(\mathcal{A}/k)$ is self-dual, e.g, 
    via the Mukai pairing 
    \cite{Sh13}; here we can more explicitly see this because $\mathcal{A}$ is
    Morita equivalent to $\mathcal{A}^{op}$ (since $2[\mathcal{A}]=0$ in
    $\Br(X)$). In
    particular, for any integer $i$, we have
    $H^i(\HH(\Ascr/k))\iso H^{-i}(\HH(\Ascr/k))$. The differential
    $d_2^\alpha\colon H^0(X,\Omega^0_X)\rightarrow H^2(X,\Omega^1_X)$ is an
    isomorphism by Proposition~\ref{prop:differential},
    Proposition~\ref{prop:enriques}, and dimension considerations (see Figure~\ref{fig:hodge}). We see from Figure~\ref{fig:hodge} that
    $H^1(\HH(\Ascr/k))=0$. Thus, $H^{-1}(\HH(\Ascr/k))=0$ and thus the differential
    $d_2^\alpha\colon H^0(X,\Omega^1)\rightarrow H^2(X,\Omega^2_X)$ must be an
    isomorphism.
\end{proof}

\begin{figure}[h]
    \centering
    \begin{tabular}{c c c c}
        &   $\Omega_X^0$   &   $\Omega^1_X$   &   $\Omega_X^2$\\[0.5ex]
        $H^2$  &   $0$ & $k$ & $k$\\[0.5ex]
        $H^1$  &   $0$ &   $k^{12}$    & $0$\\[0.5ex]
        $H^0$  &   $k$ &   $k$ &   $0$
    \end{tabular}
    \caption{The Hodge cohomology of a classical Enriques surface reproduced
    from~\cite[7.3.8]{illusie-derham-witt}.}
    \label{fig:hodge}
\end{figure}

\begin{figure}[h]
    \centering
    \begin{tabular}{c c c c}
        &   $W\Omega_X^0$   &   $W\Omega^1_X$   &   $W\Omega_X^2$\\[0.5ex]
        $H^2$  &   $0$ & $k$ & $W$\\[0.5ex]
        $H^1$  &   $0$ &   $W^{10}\oplus k$    & $0$\\[0.5ex]
        $H^0$  &   $W$ &   $0$ &   $0$
    \end{tabular}
    \caption{The Hodge--Witt cohomology of a classical Enriques surface
    reproduced from~\cite[7.3.6]{illusie-derham-witt}.}
    \label{fig:hodge-witt}
\end{figure}

\section{The conic bundle over a classical Enriques surface}\label{sec:conic}

Let $X$ be a classical Enriques surface over an algebraically closed field $k$
of characteristic $2$. Let $\alpha\in\Br(X)\iso\ZZ/2$ be the nonzero class.
Since $X$ is a surface over an algebraically closed field,
$\ind(\alpha)=\per(\alpha)$ by \cite{dJ04}, so $\alpha$ is represented by a quaternion algebra
$D$ over the generic point. Since $X$ is a regular $2$-dimensional scheme, $D$
spreads out to a degree $2$ Azumaya algebra $\Ascr$ over $X$. Let $P\rightarrow
X$ be the Severi--Brauer scheme associated to $\Ascr$; it is a $\PP^1$-bundle
locally trivial in the \'etale topology.

\begin{calculation}
    There is a semiorthogonal decomposition
    $$\Perf(P)\we\langle\Perf(X),\Perf(\Ascr)\rangle.$$ Thus, we can compute
    the Hochschild homology of $P$ using additivity. (This is basically
    Quillen's argument from~\cite[Section~9]{quillen}, but see
    also~\cite{bernardara}.) 
	 By \cite{antieau-vezzosi}, the HKR spectral sequence degenerates for $X$,
	 so (by Figure~\ref{fig:hodge})
$H^i( \HH(X/k))$ vanishes for $i \notin \left\{-1, 0, 1\right\}$; it is
one-dimensional $i = \pm 1$ and $14$-dimensional for $i = 0$. 
	 By Corollary~\ref{cor:hh}, the $k$-vector spaces
    $H^i(\HH(P/k))$ vanish for $i\notin\{-1,0,1\}$. They are one dimensional for
    $i=\pm 1$. And, for $i=0$, it is a $26$-dimensional vector space.
\end{calculation}

\begin{calculation}
    Let $\pi \colon P\rightarrow X$ be the structure morphism. Then,
    $\pi^* L_X\rightarrow L_P\rightarrow L_\pi$ is exact. We find an exact sequence
    $$0\rightarrow \pi^*\Omega^1_X\rightarrow\Omega^1_P\rightarrow\omega_{P/X}\rightarrow
    0$$ of vector bundles. Hence, we have an equivalence
    $\pi^*\Omega_X^2\otimes_{\Oscr_P}\omega_{P/X}\we\Omega^3_P$.
\end{calculation}

We might hope based on the failure of twisted HKR that HKR also fails for $P$.
However, the next result proves that in fact the HKR spectral sequence
does degenerate.

\begin{theorem}
    For the conic bundle $P$ constructed above, the HKR spectral
    sequence degenerates at $E_2$.
\end{theorem}

\begin{proof}
    We know by the calculation above the dimension of each $k$-vector space
    $H^i(\HH(P/k))$. It is enough to check that
    $$\sum_{s+t=i}\dim_kH^s(P,\Omega_P^{-t})=\dim H^i(\HH(P/k)).$$
    We can do this calculation up to a small discrepancy for any $\PP^1$-bundle
    over $X$. Resolving this discrepancy involves using the fact that the
    pullback map $\Br(X)\rightarrow\Br(P)$ kills $\alpha$.

    We will fill out the table in Figure~\ref{fig:hodge-conic} for this
    threefold. By Serre duality, saying that
    $H^s(P,\Omega_P^t)\iso H^{3-s}(P,\Omega_P^{3-t})^\ast$, it is enough to fill
    out the first two columns. In the figure, $\epsilon$ refers to a fixed
    number, either $0$ or $1$; it is the same number each place it appears.

    \begin{figure}[h]
        \centering
        \begin{tabular}{c c c c c}
            &   $\Omega_P^0$   &   $\Omega^1_P$   &   $\Omega_P^2$ & $\Omega_P^3$\\[0.5ex]
            $H^3$  &   $0$ &  $0$   &  $k$    & $k$\\[0.5ex]
            $H^2$  &   $0$ & $k^{0+\epsilon}$ & $k^{12+\epsilon}$&0\\[0.5ex]
            $H^1$  &   $0$ &   $k^{12+\epsilon}$    & $k^{0+\epsilon}$&0\\[0.5ex]
            $H^0$  &   $k$ &   $k$ &   $0$  & $0$
        \end{tabular}
        \caption{The Hodge cohomology of the conic bundle over a classical Enriques surface.}
        \label{fig:hodge-conic}
    \end{figure}

    Since $\pi\colon P\rightarrow X$ is a $\PP^1$-bundle, we find that
    $R \pi_*\Oscr_P\we\Oscr_X$. Hence,
    $H^s(P,\pi^*\Omega^t_X)\iso H^s(X,\Omega^t_X)$ by adjunction for all $s,t$.
    Moreover, we see immediately that $H^s(P,\Oscr_P)=0$ for $s>0$. This
    computes the first column. For the second, we use the conormal sequence.
    Let's first compute the cohomology of $\omega_{P/X}$. By
    Grothendieck-Verdier duality for $P\rightarrow X$, we see that $R
    \pi_*\omega_{P/X}\we\Oscr_X[-1]$. Therefore, using the spectral sequence
    $$E_2^{a,b}=H^a(X,R^b
    \pi_*\omega_{P/X})\Rightarrow H^{a+b}(P,\omega_{P/X}),$$ we see that
    $H^s(P,\omega_{P/X})\iso H^{s-1}(X,\Oscr_X)$. In other words, using
    Figure~\ref{fig:hodge}, we have that $H^s(P,\omega_{P/X})$ is
    $1$-dimensional for $s=1$ and zero otherwise. The long exact sequence for
    the cohomology of $0\rightarrow
    \pi^*\Omega_X^1\rightarrow\Omega^1_P\rightarrow\omega_{P/X}\rightarrow 0$ gives
    us an exact sequence
    $$0\rightarrow H^1(X,\Omega^1_X)\rightarrow H^1(P,\Omega^1_P)\rightarrow
    k\rightarrow H^2(X,\Omega^1_X)\rightarrow H^2(P,\Omega^1_P)\rightarrow 0$$
    as well as isomorphisms
    \begin{align*}
        H^0(X,\Omega^1_X)&\iso H^0(P,\Omega^1_P)\\
        H^3(X,\Omega^1_X)&\iso H^3(P,\Omega^1_P).
    \end{align*}
    Hence, $H^0(P,\Omega^1_P)\iso k$ and $H^3(P,\Omega^1_P)=0$.

    Using the exact sequence, we see that $\epsilon=1$ if
    $H^1(P,\Omega_P^1)\rightarrow k$ is surjective (bearing in mind that
    $H^2(X,\Omega^1_X)$ is $1$-dimensional). We have $\epsilon=0$ if
    $H^2(X,\Omega^1_X)\rightarrow H^2(P,\Omega^1_P)$ is zero.

    This completes the analysis of the table. To determine $\epsilon$, we use
    the commutative diagram
    $$\xymatrix{
        H^2(X,\Gm)\ar[r]^{\dlog}\ar[d] &   H^2(X,\Omega_X^1)\ar[d]\\
        H^2(P,\Gm)\ar[r]^{\dlog}& H^2(P,\Omega_P^1).
    }$$ Since the pullback map $H^2(X,\Gm)\rightarrow H^2(P,\Gm)$ kills
    $\alpha$, by definition of the Severi--Brauer scheme, and since
    $\dlog\alpha$ is nonzero in $H^2(X,\Omega_X^1)$, we see that
    $H^2(X,\Omega^1_X)\rightarrow H^2(P,\Omega^1_P)$ has a non-trivial kernel
    and hence is identically zero since the map is $k$-linear and
    $H^2(X,\Omega^1_X)$ is $1$-dimensional over $k$.
    Thus, we get $\epsilon = 0$ and a dimension count now completes the proof.
\end{proof}

    We see from this a phenomenon which only exists in characteristic $p>0$: 
	 Hodge cohomology can distinguish between $\PP^n$-bundles.
To begin with, we observe that this never happens in characteristic zero.

\begin{proposition}[{Hodge cohomology of $\PP^n$-bundles} ]
\label{HodgecohPnbundle}
Suppose $k$ is a field of characteristic zero and $X$ is a smooth $k$-scheme.
Let $\pi \colon P \to X$ be a Severi-Brauer scheme (i.e., an \'etale locally trivial $\PP^n$-bundle). The Hodge cohomology $H^* ( P, \Omega_P^{*})$ is a free bigraded module
over $H^*(X, \Omega_X^{\ast})$ on the set $\{1,c,...,c^n\}$, where $c = c_1(\omega_{P/X}) \in H^1(P,\Omega^1_P)$.
\end{proposition} 

This proposition may be regarded as an instance of the Leray--Hirsch theorem for Hodge cohomology. As the proof below shows, the result also holds true in characteristic $p$ provided $p \nmid (n+1)$.

\begin{proof} 
As in the proof of \Cref{prop:pbf}, because the formation of $\omega_{P/X}$ and $c$ commutes with base change on $X$, we may work \'etale locally on $X$ to reduce to the case $P = \mathbf{P}^n \times X$. By the K\"{u}nneth formula, $H^*(P, \Omega^*_P)$ is a free bigraded module over $H^*(X, \Omega^*_X)$ on the set $\{1,d,...,d^{n}\}$, where $d = c_1(\mathrm{pr}_1^* \mathcal{O}_{\mathbf{P}^n}(1))$. Since $\omega_{P/X} \simeq \mathrm{pr}_1^* \mathcal{O}_{\mathbf{P}^n}(-n-1)$, we have $c = -(n+1)d$. As we are working in characteristic $0$, it is then clear that $H^*(P, \Omega^*_P)$ is then also a free bigraded module over $H^*(X, \Omega^*_X)$ on the set $\{1,c,...,c^{n}\}$.
\end{proof} 

 For $\PP^n$-bundles that arise as
	 projectivizations of vector bundles, \Cref{HodgecohPnbundle} holds in arbitrary
	 characteristic
	 as in \Cref{prop:pbf}. 
By contrast, in 	characteristic $p> 0$, we see that
    Hodge cohomology can distinguish between $\PP^n$-bundles $P$ and $\PP^n_X$ in some cases. 
	If $X$ is a smooth proper variety over any field $k$  and $\pi \colon P \to
	X$ is a $\PP^n$-bundle, then a slight elaboration of the above proof (to
	handle $R \pi_* \Omega^{i}_{P}$) shows that 
	the Hodge cohomology of $P$ is  bounded above by that of $\PP^n_X$. However,
	this bound may be strict. 
For example, 
	 in characteristic two, we have seen above that the pullback in Hodge cohomology 
	 the conic bundle
for $P \to X$ over a classical Enriques surface $X$ is not injective;
specifically, $H^2(X, \Omega^1_X) \to H^2( P, \Omega^1_P)$ was not injective
thanks to the $\mathrm{dlog}$ map. 
In characteristic $p > 2$, 
the results of Section~\ref{sec:PGLn} show that there is a 
smooth projective threefold $Y$ and a $\mathbb{P}^{p-1}$-bundle $\pi': P' \to Y$ such
that the pullback $H^2(Y, \Omega^1_Y) \to H^2( P', \Omega^1_{P'})$ is not
injective.

\bibliographystyle{amsplain}
\bibliography{HKR3}

\end{document}